\documentclass[12pt]{amsart}
\pdfoutput=1
\usepackage[headings]{fullpage}
\usepackage{amssymb,epic,eepic,epsfig}
\usepackage{epstopdf}
\usepackage{graphicx}
\usepackage{texdraw}
\usepackage{url}
\usepackage[all]{xy}
\usepackage{float}   
\usepackage[bookmarks=true,%
    colorlinks=true,%
    linkcolor=blue,%
    citecolor=blue,%
    filecolor=blue,%
    menucolor=blue,%
    urlcolor=blue,%
    breaklinks=true]{hyperref}

\newtheorem{theorem}{Theorem}[section]
\theoremstyle{definition}
\newtheorem{proposition}[theorem]{Proposition}
\newtheorem{lemma}[theorem]{Lemma}
\newtheorem{definition}[theorem]{Definition}
\newtheorem{remark}[theorem]{Remark}
\newtheorem{corollary}[theorem]{Corollary}

\newtheorem{question}[theorem]{Question}

\def\BN{\mathbb N}
\def\BZ{\mathbb Z}

\def\BQ{\mathbb Q}
\def\BR{\mathbb R}
\def\BC{\mathbb C}

\def\calA{\mathcal A}
\def\calB{\mathcal B}

\def\calT{\mathcal T}
\def\calS{\mathcal S}

\def\a{\alpha}

\def\ga{\gamma}

\def\d{\delta}
\def\b{\beta}

\def\longto{\longrightarrow}

\def\pt{\partial}

\def\calB{\mathcal{B}}

\def\coeff{\mathrm{coeff}}

\def\be{  \begin{equation} }
\def\ee{  \end{equation} }

\def\ID{I_{\Delta}}
\def\rk{\mathrm{rk}}
\def\GA{\mathrm{GA}}
\def\SA{\mathrm{SA}}

\newcommand{\ha}{\frac12}
\def\GL{\mathrm{GL}}

\newcommand{\mb}{\mathbf}
\newcommand{\wt}{\widetilde}

\numberwithin{equation}{section}

\begin{document}

\title[The 3D index of an ideal triangulation and angle structures]{
The 3D index of an ideal triangulation and angle structures}

\author{Stavros Garoufalidis}
\address{School of Mathematics \\
         Georgia Institute of Technology \\
         Atlanta, GA 30332-0160, USA \newline
         {\tt \url{http://www.math.gatech.edu/~stavros }}}
\email{stavros@math.gatech.edu}
\thanks{Supported in part by NSF grant DMS-11-05678. 
\\
\newline
1991 {\em Mathematics Classification.} Primary 57N10. Secondary 57M25.
}

\date{October 6, 2015}
\dedicatory{\rm{With an appendix by Sander Zwegers}}

\begin{abstract}
The 3D index of Dimofte-Gaiotto-Gukov a partially defined function
on the set of ideal triangulations of 3-manifolds with $r$ torii 
boundary components. For a fixed $2r$ tuple of integers, 
the index takes values in the set of $q$-series with integer coefficients.

Our goal is to give an axiomatic definition of the tetrahedron index,
and a proof that the domain of the 3D index consists precisely of the set of 
ideal triangulations that support an index structure. The latter is
a generalization of a strict angle structure. We also prove that the 3D index 
is invariant under 3-2 moves, but not in general under 2-3 moves.
\end{abstract}

\maketitle

\tableofcontents

\section{Introduction}
\label{sec.intro}


In a series of papers \cite{DGG1,DGG2}, Dimofte-Gaiotto-Gukov studied
topological gauge theories using as input an ideal triangulation $\calT$ of 
a 3-manifold $M$. These gauge theories 
play an important role in 
\begin{itemize}
\item Chern-Simons 
perturbation theory (that fits well with the earlier work on
quantum Riemann surfaces of \cite{D1} and the later work on the perturbative 
invariants of \cite{DG}),
\item
categorification and Khovanov Homology, that fits with
the earlier work \cite{Witten-fivebranes}.
\end{itemize}
Although the gauge theory depends on the ideal triangulation $\calT$, 
and the 3D index in general may not converge, physics predicts that the 
gauge theory ought to be a topological invariant of the 
underlying 3-manifold $M$. When $\pt M$ consists of $r$ torii, the low energy 
description of these gauge theories gives rise to a {\em partially} defined 
function
\be
\label{eq.ICT}
I: \{\text{ideal triangulations}\}
\longto \BZ((q^{1/2}))^{\BZ^r \times \BZ^r}, \qquad \calT \mapsto 
I_{\calT}(m_1,\dots,m_r,e_1,\dots,e_r) \in \BZ((q^{1/2}))
\ee
for integers $m_i$ and $e_i$, which is invariant under some {\em partial} 
2-3 moves. The building block of the 3D index $I_{\calT}$
is the {\em tetrahedron index $\ID(m,e)(q) \in \BZ[[q^{1/2}]]$} defined by
\footnote{The variables $(m,e)$ are named after the magnetic and electric 
charges of \cite{DGG2}.}
\be
\label{eq.ID}
\ID(m,e)=\sum_{n=(-e)_+}^\infty (-1)^n \frac{q^{\frac{1}{2}n(n+1)
-\left(n+\frac{1}{2}e\right)m}}{(q)_n(q)_{n+e}}
\ee
where
$$
e_+=\max\{0,e\}
$$
and $(q)_n=\prod_{i=1}^n (1-q^i)$. If we wish, we can sum in the above 
equation over the integers, with the understanding that $1/(q)_n=0$ for $n < 0$.

Roughly, the 3D index $I_{\calT}$ of an ideal triangulation $\calT$ is a sum over
tuples of integers of a finite product of tetrahedron indices evaluated at
some linear forms in the summation variables. Convergence of such sums
is not obvious, and not always expected on physics grounds. For instance,
the following sum
$$
\sum_{e \in \BZ} \ID(0,e) q^{v e}
$$
converges in $\BZ((q^{1/2}))$ if and only if $v >0$.
This follows easily from the fact that the degree $\d(e)$ of the summand 
is given by
$$
\d(e)=e + \begin{cases} 0 & \text{if} \,\,  e \geq 0 \\
\frac{e^2}{2}-\frac{e}{2} & \text{if} \,\,  e \leq 0
\end{cases}
$$ 
Our goal is to
\begin{itemize}
\item[(a)]
prove that the 3D index $I_{\calT}$ exists if and only if $\calT$
admits an index structure (a generalization of a strict angle structure); 
see Theorem \ref{thm.3} below.
\item[(b)]
give a complete axiomatic definition of the tetrahedron index
$\ID$ focusing on the combinatorial and $q$-holonomic aspects; see
Section \ref{sec.axioms}.
\item[(c)]
to show that the 3D index is invariant under $3 \to 2$ moves, but not
in general under $2 \to 3$ moves, and give a necessary and sufficient
criterion for invariance under $2 \leftrightarrow 3$ moves; see Section 
\ref{sec.23}.
\end{itemize}


\section{Index structures, angle structures and the 3D index}
\label{sec.results}

\subsection{Index structures}
\label{sub.index}

Consider two $r \times s$ matrices $\mb A$ and $\mb B$ with integer entries
and a column vector $v \in \BZ^r$, and let $\mb M=(\mb A | \mb B | v)$.

\begin{definition}
\label{def.strictindex}
\rm{(a)} We say that $\mb M$ supports an {\em index structure} if the rank 
of $(\mb A | \mb B)$ is $r$ and for every $Q: \{1,\dots,s\}\to\{1,2,3\}$
there exists $(\a,\b,\ga) \in \BQ^{3s}$ that satisfies
\be
\label{eq.index.structure}
\mb A \a + \mb B \ga = \nu, \qquad
\a+\b+\ga=(1,\dots,1)^T
\ee
and $Q(\a,\b,\ga) >0$. The latter means that for every $i=1,\dots,s$
the following inequalities are satisfied:
\be
\label{eq.index.ineq}
\begin{cases}
\a_i > 0 & \text{if} \quad Q(i)=1 \\
\b_i > 0 & \text{if} \quad Q(i)=2 \\
\ga_i > 0 & \text{if} \quad Q(i)=3 \\
\end{cases}
\ee
\rm{(b)} We say that $\mb M$ supports a {\em strict index structure} 
if the rank of $(\mb A | \mb B)$ is $r$ and there exists $(\a,\b,\ga) \in
\BQ_+^{3s}$ that satisfies \eqref{eq.index.structure}, where $\BQ^+$ is the
set of positive rational numbers.
\end{definition}  
It is easy to see that if $\mb M$ supports a strict index structure, then
it supports an index structure, but not conversely. As we will see in Section
\ref{sub.angle}, ideal triangulations $\calT$ give rise to matrices $\mb M$,
and a strict index structure on $\mb M$ is a strict angle structure on $\calT$.
On the other hand, index structures are new and motivated by Theorem 
\ref{thm.IMconv} below.

The next definition discusses two actions on $\mb M$: an action of $\GL(r,\BZ)$
on the left which allows for row operations on $\mb M$, and a cyclic action
of order three at the pair of the $i$th columns of $(\mb A|\mb B)$. 

\begin{definition}
\label{def.GLr}
\rm{(a)}
There is a left action of $\GL(r,\BZ)$ on $\mb M$, defined by
$$
P  \in \GL(r,\BZ) \qquad \mb M=(\mb A|\mb B|v)
\qquad P \mb{M}=(P \mb A|P \mb B|P v)
$$ 
An index structure on $\mb M$ is also an
index structure on $P \mb M$.
\newline
\rm{(b)} There is a left action of $(\BZ/3)^s$ on $\mb M$ acting
on the $i$th columns $(a_i|b_i)$ of $(\mb A|\mb B)$ (and fixing
all other columns) given by
\be
\label{eq.Z3M}
(a_i|b_i|v) \overset{S}{\mapsto} (-b_i | a_i -b_i|v-b_i)\,,
\ee
where 
\be
\label{eq.defS}
S(a|b|v)=(-b|a-b|v-b)
\ee
satisfies $S^3=\mathrm{Id}$. We extend $S$ to act on an index structure
$(\a,\b,\ga)$ of $\mb M$ by 
\be
\label{eq.Z3angle}
(\a_i,\b_i,\ga_i) \overset{S}{\mapsto} (\b_i,\ga_i,\a_i)
\ee
and fixing all other coordinates of $(\a,\b,\ga)$. It is easy to see that 
if $(\a,\b,\ga)$ is an index structure on $\mb M$ and $S \in (\BZ/3)^s$,
then $S(\a,\b,\ga)$ is an index structure of $S \mb M$.
\end{definition}

\begin{definition}
\label{def.IM}
Given $\mb M$, and  $m=(m_1,\dots,m_s), e=(e_1,\dots,e_s) 
\in \BZ^s$ consider the sum
\be
\label{eq.IM}
I_{\mb M}(m,e)(q)=\sum_{k \in \BZ^r}
q^{\frac{1}{2} v \cdot k}
\prod_{i=1}^s \ID(m_i-b_i \cdot k, e_i+a_i \cdot k)
\ee
\end{definition}

\begin{theorem}
\label{thm.IMconv}
$I_{\mb M}(m,e)(q) \in \BZ((q^{1/2}))$ is convergent for all 
$m, e\in \BZ^s$ if and only if $\mb M$ supports an index structure.
In that case, $I_{\mb M}$ is $q$-holonomic in the variables $(m,e)$.
\end{theorem}

\begin{remark}
\label{rem.WZ}
$q$-holonomicity in Theorem \ref{thm.IMconv} follows immediately from 
\cite{WZ}. Convergence is the main difficulty.
\end{remark}

\begin{remark}
\label{rem.nahm}
By definition, $I_{\mb M}$ is a generalized Nahm sum in the sense of
\cite{GL2}, where the summation is over a lattice.
\end{remark}

\begin{corollary}
\label{cor.convP}
Applying Theorem \ref{thm.IMconv} to the case $r=1$, $s=3$,
$\mb M=(\mb A| \mb B|v)=(1 \, 1 \, 1|0 \, 0 \, 0|2)$ and the
strict index structure $2=\frac{2}{3}\cdot 1 + \frac{2}{3}\cdot 1 +
\frac{2}{3}\cdot 1$, it follows that
the right hand side of the pentagon identity \eqref{eq.pentagon}
is convergent in $\BZ((q^{1/2}))$.
\end{corollary}

The next remark discusses the invariance of the index under the actions
of Definition \ref{def.GLr}.

\begin{remark}
\label{rem.GLr}
Fix $\mb M$ that supports an index structure. Then, for $P \in \GL(r,\BZ)$ and 
$S \in (\BZ/3)^s$, it follows that $P \mb M$ and $S \mb M$ also supports
an index structure. In that case, Theorem \ref{thm.IMconv}
implies that $I_{\mb M}$, $I_{P \mb M}$ and $I_{S \mb M}$ are all convergent.
We claim that
$$
I_{P \mb M} = I_{\mb M}, \qquad I_{S \mb M} = I_{\mb M} \,.
$$
The first equality follows by changing variables $k \mapsto P k$ in the
definition of $I_{\mb M}$ given by \eqref{eq.IM}. The second equality follows
from the fact that the tetrahedron index $\ID$ satisfies Equation
\eqref{eq.Z3}; this is shown in part (a) of  Theorem \ref{thm.2}.
\end{remark}

The next corollary follows easily from Theorem \ref{thm.IMconv} and the
definition of an index structure on $(\mb A|0|v_0)$.

\begin{corollary}
\label{cor.Aconv}
Fix an $r \times s$ matrix $\mb A$ with integer entries and columns
$v_i$ for $i=1,\dots,s$, and let $v_0 \in \BZ^r$, and 
let $\mb M=(\mb A|0|v_0)$. The following are equivalent:
\begin{itemize}
\item[(a)] $I_{\mb M}(q)$ converges.
\item[(b)] $\rk(\mb A)=r$ and there exist $\a_i >0$ for $i=1,\dots,s$ 
such that $v_0=\sum_{i=1}^s \a_i v_i$.
\end{itemize}
\end{corollary}

\begin{question}
\label{que.compare}
Compare the $q$-series $I_{(\mb A|0|v)}$ with the vector
partition functions of Sturmfels \cite{Sturmfels} and Brion-Vergne \cite{BV},
and the $q$-hypergeometric systems of 
equations of \cite{SST}. 
\end{question}

\subsection{Angle structures}
\label{sub.angle}

In this section we define the 3D index of an ideal triangulation.  
A {\em generalized angle structure} on a combinatorial ideal 
tetrahedron $\Delta$ is an assignment of real numbers (called {\em angles}) 
at each edge of $\Delta$ such that the sum of the three angles around each 
vertex is $1$.\footnote{The sum of the 3 angles around each vertex is 
traditionally $\pi$.} 
It is easy to see that opposite edges are assigned the same
angle, thus a generalized angle structure is determined by a triple $(\a,\b,\ga)
\in \BR^3$ that satisfies $\a+\b+\ga=1$; see Figure \ref{fig.angles}.

\begin{figure}[htpb]
\begin{center}
\includegraphics[height=0.15\textheight]{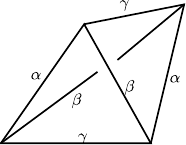}
\caption{Angles of a tetrahedron.}
\label{fig.angles}
\end{center}
\end{figure}

A generalized angle structure is {\em strict} if $\a,\b,\ga >0$.
Let $\calT$ denote an ideal triangulation of an oriented 
3-manifold $M$ with torus boundary. A {\em generalized angle structure} 
on $\calT$ 
is the assignment of angles at each tetrahedron of $\calT$ such that
the sum of angles around every edge of $\calT$ is $2$. A generalized 
angle structure
on $\calT$ is {\em strict} if its restriction to each tetrahedron is strict.
For a detailed 
discussion of angle structures and their duality with normal surfaces,
see  \cite{HRS,LT,Ti}. Generalized angle structures are linearizations
of the gluing equations, that may be used to construct complete hyperbolic
structures, and intimately connected with the theory of normal surfaces
on $M$ \cite{Jaco}.

The existence of a strict angle structure imposes restrictions on the
topology of $M$: it implies that $M$ is irreducible, atoroidal and each
boundary component of $M$ is a torus; see for example \cite{LT}. On the
other hand, if $M$ is a hyperbolic link complement, then there exist 
triangulations which admit a strict angle structure, \cite{HRS}. In fact, such
triangulations can be constructed by a suitable refinement of the 
Epstein-Penner ideal cell decomposition of $M$. Note that not
all such triangulations are geometric \cite{HRS}.

\subsection{The Neumann-Zagier matrices}
\label{sub.NZ}

Fix is an oriented ideal triangulation $\calT$ with $N$ tetrahedra
of a 3-manifold $M$ with torii boundary components. Sssign variables
$Z_i,Z'_i,Z''_i$ at the opposite edges of each tetrahedron $\Delta_i$ 
respecting its orientation as in Figure \ref{fig.tetZZZ}.

\begin{figure}[htpb]
\begin{center}
\includegraphics[height=0.15\textheight]{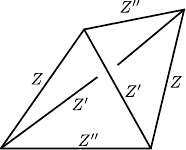}
\caption{Shapes of a tetrahedron.}
\label{fig.tetZZZ}
\end{center}
\end{figure}

Then we can read off matrices $N\times N$
matrices $\bar{\mb A}$, $\bar{\mb B}$ and $\bar{\mb C}$ whose rows are 
indexed by the $N$ edges of $\calT$ and whose columns are indexed by the 
$Z_i,Z'_i,Z''_i$ variables. These are the so-called 
{\em Neumann-Zagier matrices} that
encode the exponents of the {\em gluing equations} of $\calT$, originally
introduced by Thurston \cite{NZ,Th}. In terms of these matrices, 
a generalized angle structure is
a triple of vectors $\a,\b,\ga \in \BR^N$ that satisfy the equations
\be
\label{eq.angle}
\bar{\mb A} \a + \bar{\mb B} \b + \bar{\mb C} \ga = (2,\dots,2)^T, 
\qquad \a+\b+\ga=(1,\dots,1)^T \,.
\ee
A {\em quad} $Q$ for $\calT$ is a choice of pair of opposite edges 
at each tetrahedron $\Delta_i$ for $i=1,\dots,N$. $Q$ can be used to eliminate
one of the three variables $\a_i,\b_i,\ga_i$ at each tetrahedron using the
relation $\a_i+\b_i+\ga_i=1$. Doing so, Equations \eqref{eq.angle} take
the form
$$
\mb A \a + \mb B \ga = \nu \,.
$$ 
The matrices $(\mb A|\mb  B)$ have some key 
{\em symplectic properties}, discovered by Neumann-Zagier when $M$ is a 
hyperbolic 3-manifold (and $\calT$ is well-adapted to the hyperbolic structure) 
\cite{NZ}, and later generalized to the case of arbitrary 3-manifolds 
in \cite{Neumann-combi}. Neumann-Zagier show that the rank of 
$(\mb A|\mb  B)$ is $N-r$, where $r$ is the number of boundary 
components of $M$; all assumed torii. If we choose $N-r$ linearly
independent rows of $(\mb A|\mb  B)$, then we obtain matrices 
$(\mb A'|\mb B')$ and a vector $\nu'$, which combine to 
$\mb M=(\mb A'|\mb B'|\nu')$. In addition, the exponents of meridian and
longitude loops at each boundary torus give additional matrices
$(a^T,b^T)$ and $(c^T,d^T)$ of size $r \times 2N$.

\begin{definition}
\label{def.indexT}
The 3D index of $\calT$ is defined by  
\be
\label{eq.index2}
I_{\calT}(m,e)(q)=I_{\mb M}(d m - b e, -c m + a e)(q)
\ee
\end{definition}
Implicit in the above definition is a choice of quad $Q$ and a choice
of rows to remove. However, the index is independent of these choices;
see Remark \ref{rem.GLr}. Keep in mind the action of $(\BZ/3)^N$ given by
acting on the $i$th columns $\bar a_i$, $\bar b_i$ and $\bar c_i$
of $\bar{\mb A}$, $\bar{\mb B}$ and $\bar{\mb C}$ by 
$$
S(\bar a_i |\bar b_i |\bar c_i)=(\bar b_i |\bar c_i |\bar a_i), \qquad
$$
(and fixing all other columns) and on the $i$th coordinates of an angle 
structure by
$$
S(\a_i,\b_i,\ga_i)=(\b_i,\ga_i,\a_i)
$$
(and fixing all other coordinates) 
and on the $i$th columns $a_i$ and $b_i$ of $\mb A$ and $\mb B$ by
$$
S(a_i|b_i|\nu)=(-b_i |a_i-b_i|\nu-b_i) \,.
$$
(and fixing all other columns). Since the rank of $(\mb A|\mb B)$ is $N-r$
and $\mb A, \mb B$ are $(N-r)\times N$ matrices, it follows that $\mb M$
admits a strict structure if and only $\calT$ admits a strict
angle structure. In addition, $\calT$ admits an index structure if for every
choice of quad $Q$ there exist a solution $(\a,\b,\ga)$ of Equations
\eqref{eq.angle} that satisfies the inequalities \eqref{eq.index.ineq}.
Theorem \ref{thm.IMconv} implies the following.

\begin{theorem}
\label{thm.3}
The index $I_{\calT}: \BZ^r \times \BZ^r \longto \BZ((q^{1/2}))$
is well-defined if and only if $\calT$ admits an index structure. In particular,
$I_{\calT}$ exists if $\calT$ admits a strict angle structure.
\end{theorem}

See Section \ref{sub.m136} for an example of an ideal triangulation $\calT$ of
the census manifold {\tt m136} \cite{census} which admits a semi-strict 
angle structure (i.e., angles are nonnegative real numbers), does not
admit a strict angle structure, and which has a solution of the gluing 
equations that recover the complete hyperbolic structure. A case-by-case
analysis shows that this example admits an index structure, thus the index
$I_{\calT}$ exists. This example
appears in \cite[Example 7.7]{HRS}. We thank H. Segerman for a detailed
analysis of this example. 

\subsection{On the topological invariance of the index}
\label{sub.invariance}

Physics predicts that when defined, the 3D index $I_{\calT}$ depends only
on the underlying 3-manifold $M$. Recall that \cite{HRS} prove that every
hyperbolic 3-manifold $M$ that satisfies
\be
\label{eq.H12}
H_1(M,\BZ/2)\to H_1(M,\pt M,\BZ/2) \qquad \text{is the zero map}
\ee
(eg. a hyperbolic link complement) admits an ideal triangulation with
a strict angle structure, and conversely if $M$ has an ideal triangulation
with a strict angle structure, then $M$ is irreducible, atoroidal and 
every boundary component of $M$ is a torus \cite{LT}.

A simple way to construct a topological invariant using the index, would
be a map
$$
M \mapsto \{I_{\calT} \,\,|\,\, \calT \in \calS_M \}
$$
where $M$ is a cusped hyperbolic 3-manifold with at least one cusp and
$\calS_M$ is the set of ideal triangulations of $M$ that support an index
structure. The latter is a nonempty (generally infinite) set by \cite{HRS},
assuming that $M$ satisfies \eqref{eq.H12}.
If we want a finite set, we can use the subset $\calS^{\mathrm{EP}}_M$
of ideal triangulations $\calT$ of $M$ which are a refinement of the 
Epstein-Penner cell-decomposition of $M$. Again, \cite{HRS} implies that
$\calS^{\mathrm{EP}}_M$ is nonempty assuming \eqref{eq.H12}. But really,
we would prefer a single 3D index for a cusped manifold $M$, rather
than a finite collection of 3D indices.

It is known that every two combinatorial ideal triangulations of a 
3-manifold are related by a sequence of {\em 2-3 moves} \cite{Ma1,Ma2,Pi}. 
Thus, topological invariance of the 3D index follows from invariance
under 2-3 moves. 

Consider two ideal triangulations $\calT$ and $\wt{\calT}$ with $N$ and
$N+1$ tetrahedra related by a $2-3$ {\em move} shown in Figure \ref{fig.23}.

\begin{figure}[htpb]
\begin{center}
\includegraphics[height=0.20\textheight]{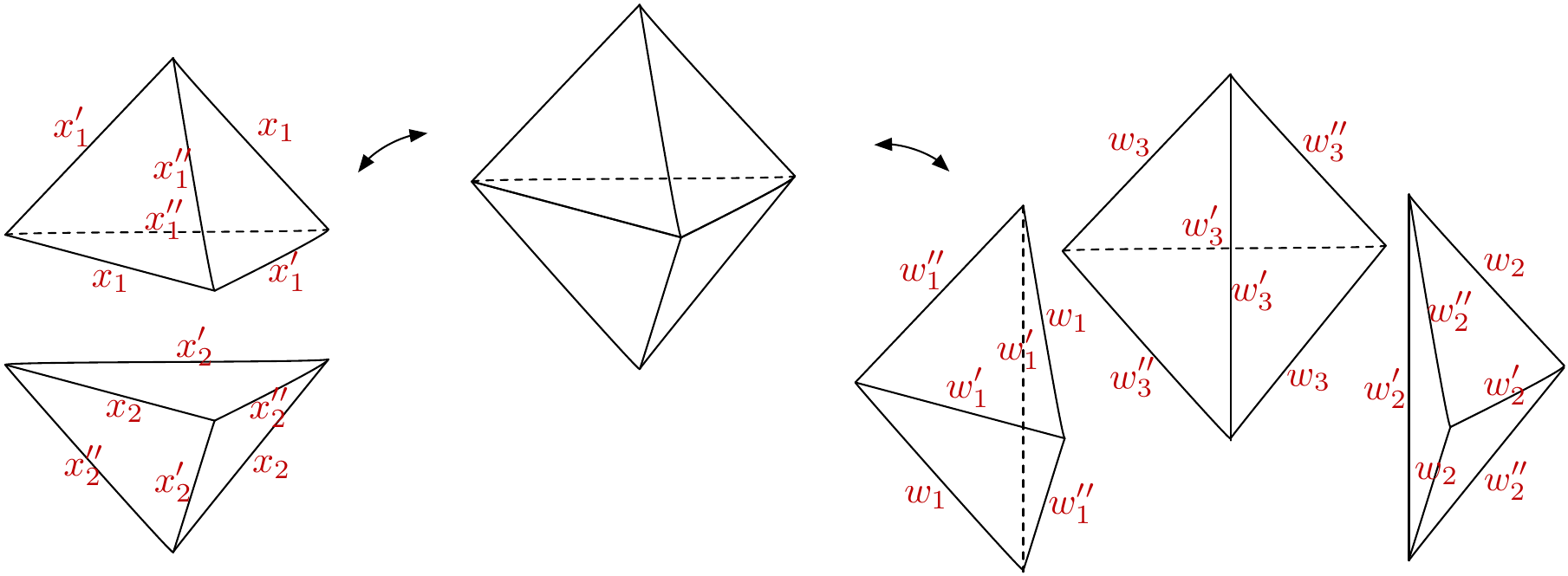}
\caption{A 2--3 move: a bipyramid split into $N$ 
tetrahedra for $\calT$ and $N+1$ tetrahedra for $\wt{\calT}$.}
\label{fig.23}
\end{center}
\end{figure}

\begin{proposition}
\label{prop.3.2.index}
If $\widetilde{\calT}$ admits a strict angle structure
structure, so does $\calT$ and $I_{\wt{\calT}}=I_{\calT}$.
\end{proposition}

For the next proposition, a special index 
structure on $\calT$ is given in Definition \ref{def.special.index}.

\begin{proposition}
\label{prop.2.3.index}
If $\calT$ admits a special strict angle structure, then 
$\wt{\calT}$ admits a strict angle structure and $I_{\wt{\calT}}=I_{\calT}$.
\end{proposition}

\begin{remark}
\label{rem.23curious}
The asymmetry in Propositions \ref{prop.3.2.index} and \ref{prop.3.2.index} 
is curious, but also necessary. The origin of this asymmetry
is the fact that 3-2 moves always preserve strict angle structures but 
2-3 moves sometimes do not. If 2-3 moves always preserved strict angle
structures, then all ideal triangulations of a fixed manifold would admit
strict angle structures as long as one of them does. On the other hand,
an ideal triangulation that contains an edge which belongs to
exactly one (or two) ideal tetrahedra does not admit a strict angle structure
since the angle equations around that edge should add to $2$. Such 
triangulations are easy to construct, even for hyperbolic 3-manifolds (eg.
the $4_1$ knot).
\end{remark}


\section{Axioms for the tetrahedron index}
\label{sec.axioms}

In this section we discuss an axiomatic approach to the tetrahedron 
index. Let $\BZ((q^{1/2}))$ (resp., $\BZ[[q^{1/2}]]$) denote the ring of 
series of the form
$$
f(q)=\sum_{n \in \frac{1}{2}\BZ} a_{n} q^{n}
$$
where there exists $n_0=n_0(f)$ such that $a_n=0$ for all $n < n_0$ (resp.,
$n < 0$). For $f(q) \in \BZ((q^{1/2}))$, its {\em degree}
$\d(f(q))$ is the largest half-integer (or infinity) such that
$f(q) \in q^{d(f)}\BZ[[q^{1/2}]]$. We will say that $f(q) \in \BZ((q^{1/2}))$
is $q$-{\em positive} if $\d(f(q)) \geq 0$.

\begin{definition}
\label{def.2rec}
A {\em tetrahedron index} is a function $f: \BZ^2 \longto \BZ((q^{1/2}))$
that satisfies the equations
\begin{subequations} 
\be
\label{eq.rec1}
 q^{\frac{e}{2}} f(m + 1, e) + q^{-\frac{m}{2}} f(m, e + 1) - f(m, e)=0
\ee
\be
\label{eq.rec2} 
 q^{\frac{e}{2}} f(m - 1, e) + q^{-\frac{m}{2}} f(m, e - 1) - f(m, e)=0
\ee
\end{subequations}
for all integers $m,e$, together with the parity condition $f(m,e) \in
q^{\frac{e m}{2}}\BZ((q))$ for all $m$ and $e$.
Let $V$ denote the set of all tetrahedron indices, and $V_+$ denote 
the set of all $q$-positive tetrahedron indices.
\end{definition}

\begin{theorem}
\label{thm.1}
\rm{(a)} $V$ is a free $q$-holonomic $\BZ((q))$-module of rank $2$.
\newline
\rm{(b)} $V_+$ is a free $q$-holonomic $\BZ[[q]]$-module of rank $1$.
\newline
\rm{(c)} If $f \in V$, then it satisfies the equation
\be
\label{eq.Z3}
f(m,e)(q)=(-q^{\frac{1}{2}})^{-e}f(e,-e-m)(q)
=(-q^{\frac{1}{2}})^m f(-e-m,m)(q)
\ee
for all integers $m$ and $e$.
\newline
\rm{(d)} If $f \in V$, then it satisfies the equations

\begin{subequations} 
\be
\label{eq.rec1a}
f(m, e+1) + (q^{e + \frac{m}{2}} - q^{-\frac{m}{2}} - q^{\frac{m}{2}}) 
f(m, e) + f(m, e - 1)  = 0
\ee
\be
\label{eq.rec2a} 
f(m+1, e) + (q^{- \frac{e}{2} - m} - q^{-\frac{e}{2}} - q^{\frac{e}{2}}) 
f(m, e) +  f(m - 1, e) =   0 
\ee
\end{subequations}
for all integers $m,e$. 
\end{theorem}

\begin{question}
\label{que.basisV}
What is a basis for $V$?
\end{question}

\begin{remark}
\label{rem.thm1a}
The proof of part (a) of Theorem \ref{thm.1} implies that if $f(m,e)$
is a tetrahedron index, then $f(m,e)$ is a a unique $\BZ[q^{\pm 1/2}]$ 
linear combination of $A$ and $B$ where $(f(0,0),f(0,1))=(A,B)$.
For example, if $C=(f(m,e))_{-2 \leq m,e \leq 2}$, then 
$C=M_A A + M_B B$ where 
{\tiny
\begin{eqnarray*}
M_A &=& 
\left(
\begin{array}{ccccc}
 1-\frac{1}{q^3}+\frac{1}{q^2}+\frac{1}{q}-q^2 & \frac{1}{q}-q & -1 & -\frac{1}{q} & -\frac{1}{q^2}+\frac{1}{q} \\
 1-\frac{1}{q^2}+\frac{1}{q} & \frac{1}{\sqrt{q}} & 0 & -\frac{1}{\sqrt{q}} & -\frac{1}{q} \\
 1-\frac{1}{q} & 1 & 1 & 0 & -1 \\
 -1 & 0 & 1 & \frac{1}{\sqrt{q}} & \frac{1}{q}-q \\
 -q & -1 & 1-\frac{1}{q} & 1-\frac{1}{q^2}+\frac{1}{q} & 1-\frac{1}{q^3}+\frac{1}{q^2}+\frac{1}{q}-q^2 \\
\end{array}
\right)
\\
M_B &=& 
\left(
\begin{array}{ccccc}
 \frac{1}{q^3}-\frac{2}{q^2}-\frac{1}{q}+q+2 q^2-q^3 & 1-\frac{1}{q}+2 q-q^2 & 2-q & -1+\frac{1}{q} & \frac{1}{q^2}-\frac{2}{q} \\
 -1+\frac{1}{q^2}-\frac{2}{q}+q & -\frac{1}{\sqrt{q}}+\sqrt{q} & 1 & \frac{1}{\sqrt{q}} & -1+\frac{1}{q} \\
 -2+\frac{1}{q} & -1 & 0 & 1 & 2-q \\
 1-q & -\sqrt{q} & -1 & -\frac{1}{\sqrt{q}}+\sqrt{q} & 1-\frac{1}{q}+2 q-q^2 \\
 2 q-q^2 & 1-q & -2+\frac{1}{q} & -1+\frac{1}{q^2}-\frac{2}{q}+q & \frac{1}{q^3}-\frac{2}{q^2}-\frac{1}{q}+q+2 q^2-q^3 \\
\end{array}
\right)
\end{eqnarray*}
}
\end{remark}

\begin{remark}
\label{rem.thm1b}
The proof of part (b) of Theorem \ref{thm.1} implies that if $f(m,e)$
is a tetrahedron index, then $f(m,e)$ is uniquely determined by
$f(0,0)=\sum_{n=0}^\infty a_n q^n$. In particular, if 
$f(0,1)=\sum_{n=0}^\infty b_n q^n$, then $b_n$ are $\BZ$-linear combinations
of $a_k$ for $k \leq n$. For example, we have:
{\small
\begin{eqnarray*}
b_{0} &=&  a_{0}
\\ 
b_{1} &=&  a_{0}+a_{1}
\\ 
b_{2} &=&  2 a_{0}+a_{1}+a_{2}
\\ 
b_{3} &=&  4 a_{0}+2 a_{1}+a_{2}+a_{3}
\\ 
b_{4} &=&  9 a_{0}+4 a_{1}+2 a_{2}+a_{3}+a_{4}
\\ 
b_{5} &=&  20 a_{0}+9 a_{1}+4 a_{2}+2 a_{3}+a_{4}+a_{5}
\\ 
b_{6} &=&  46 a_{0}+20 a_{1}+9 a_{2}+4 a_{3}+2 a_{4}+a_{5}+a_{6}
\\ 
b_{7} &=&  105 a_{0}+46 a_{1}+20 a_{2}+9 a_{3}+4 a_{4}+2 a_{5}+a_{6}+a_{7}
\\ 
b_{8} &=&  242 a_{0}+105 a_{1}+46 a_{2}+20 a_{3}+9 a_{4}+4 a_{5}+2 a_{6}+a_{7}+a_{8}
\\ 
b_{9} &=&  557 a_{0}+242 a_{1}+105 a_{2}+46 a_{3}+20 a_{4}+9 a_{5}+4 a_{6}+2 a_{7}+a_{8}+a_{9}
\\ 
b_{10} &=&  1285 a_{0}+557 a_{1}+242 a_{2}+105 a_{3}+46 a_{4}+20 a_{5}+9 a_{6}+4 a_{7}+2 a_{8}+a_{9}+a_{10}
\\ 
b_{11} &=&  2964 a_{0}+1285 a_{1}+557 a_{2}+242 a_{3}+105 a_{4}+46 a_{5}+20 a_{6}+9 a_{7}+4 a_{8}+2 a_{9}+a_{10}+a_{11}
\\ 
b_{12} &=&  6842 a_{0}+2964 a_{1}+1285 a_{2}+557 a_{3}+242 a_{4}+105 a_{5}+46 a_{6}+20 a_{7}+9 a_{8}+4 a_{9}+2 a_{10} \\ & & +a_{11}+a_{12}
\end{eqnarray*}
}
In fact, it appears that $b_n$ is a $\BN$-linear combination of
 $a_k$ for $k \leq n$, although we do not know how to show this, nor
do we know of a geometric significance of this experimental fact.
\end{remark}

The next lemma computes the degree of the tetrahedron index.

\begin{lemma}
\label{lem.degID}
The degree $\d(m,e)$ of $\ID(m,e)(q)$ is given by:
\be
\label{eq.degID}
\d(m,e) = \frac{1}{2}\left( m_+(m+e)_+ + (-m)_+ e_+ + (-e)_+(-e-m)_+
+\max\{0,m,-e\} \right) 
\ee
It follows that $\d(m,e)$ is a piece-wise quadratic polynomial given by
Figure \ref{fig.degI}.
\end{lemma}

\begin{figure}[htpb]
\begin{center}
\includegraphics[height=0.15\textheight]{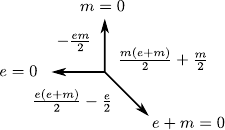}
\caption{The degree of the tetrahedron index.}
\label{fig.degI}
\end{center}
\end{figure}

The next theorem gives an axiomatic characterization of the tetrahedron
index $\ID$.

\begin{theorem}
\label{thm.2}
$\ID$ is uniquely characterized by the following equations:
\begin{itemize}
\item[(a)]
$\ID \in V_+$, $\ID(0,0)(0) \neq 0$
\item[(b)] 
$\ID$ satisfies the pentagon identity
\be
\label{eq.pentagon}
\ID(m_1-e_2,e_1)\ID(m_2-e_1,e_2)=\sum_{e_3 \in \BZ}
q^{e_3} \ID(m_1,e_1+e_3)\ID(m_2,e_2+e_3)\ID(m_1+m_2,e_3)\,,
\ee
for all integers $m_1,m_2,e_1,e_2$.
\end{itemize}
\end{theorem}

\begin{remark}
\label{rem.thm2}
The uniqueness part of Theorem \ref{thm.2} uses only the facts that
$\ID \in V$, $\d(\ID(0,e)) \geq 0$ for all $e$ and $\ID$ satisfies
the special pentagon 
$$
\ID(0,0)^2=\sum_{e \in \BZ} \ID(0,e)^3 q^{e} \,.
$$
\end{remark}


\section{Properties of a tetrahedron index}
\label{sec.properties}

\subsection{Part (d) of Theorem \ref{thm.1}}
\label{sub.partd}

Consider a function $f(m,e)$ of two discrete integer 
variables $e,m$ which satisfies Equations
\eqref{eq.rec1} and \ref{eq.rec2}. An application of the 
{\tt HolonomicFunctions.m} computer algebra package \cite{Kou} implies that
$f(m,e)$ also satisfies equations \eqref{eq.rec1a} and \eqref{eq.rec2a}.

\subsection{The rank of $V$:  part (a) of Theorem \ref{thm.1}}
\label{sub.rankV}

An application of the {\tt HolonomicFunctions.m} computer algebra package 
\cite{Kou} implies that the linear $q$-difference operators corresponding 
to the recursions of Equations \eqref{eq.rec1} and \eqref{eq.rec2}
is a Gr\"obner basis and the corresponding module has rank $2$. Said
differently, $f(m,e)$ is a unique $\BZ[q^{\pm 1/2}]$ linear combination of 
$A$ and $B$ where $A=f(0,0)$ and $B=f(0,1)$.

\subsection{The rank of $V_+$:  part (b) of Theorem \ref{thm.1}}
\label{sub.rankV+}

Consider a function $f(m,e)$ of two discrete integer 
variables $e,m$ which satisfies Equations
\eqref{eq.rec1} and \ref{eq.rec2}. Section \ref{sub.partd} implies 
that $f(0,e)$ satisfies the 3-term recursion
\be
\label{eq.recI2}
f(0, e) - (2 - q^{e-1}) f(0, e - 1) + f(0, e - 2)=0
\ee
for all integers $e$. It follows that for every integer $e$,
$f(0,e)$ is a $\BZ[q^{\pm 1}]$-linear combination of $A$ and $B$ where
$f(0,0)=A$ and $f(0,1)=B$. An induction on $e<0$ using the recursion 
relation \eqref{eq.recI2} shows that for all $e<0$ we have
$$
f(0,e) = q^{-\frac{e^2}{2}-\frac{e}{2}} \left( p_1(e) A + p_2(e) B \right)
$$
where $p_1(e), p_2(e) \in \BZ[q]$ are polynomials of maximum $q$-degree 
$e^2/2+e/2$ and constant term $(-1)^{e-1}$ and $(-1)^e$ respectively.
For example, we have:

\begin{eqnarray*}
f(0,-1)&=& A-B \\
q f(0,-2) &=&
A (-1+q) + B (1-2 q)
\\
q^3 f(0,-3) &=&
A \left(1-q-2 q^2+q^3\right) + B \left(-1+2 q+2 q^2-3 q^3\right)
\\
q^6 f(0,-4) &=&
A \left(-1+q+2 q^2+q^3-2 q^4-3 q^5+q^6\right) \\ & & 
+B \left(1-2 q-2 q^2+q^3+4 q^4+3 q^5-4 q^6\right)
\\
q^{10} f(0,-5) &=&
A \left(1-q-2 q^2-q^3+5 q^5+3 q^6+q^7-3 q^8-4 q^9+q^{10}\right) \\ & & 
+B \left(-1+2 q+2 q^2-q^3-2 q^4-7 q^5+3 q^7+6 q^8+4 q^9-5 q^{10}\right)
\end{eqnarray*}
Let us write
$$
A=\sum_{n=0}^\infty a_n q^n  \qquad
B=\sum_{n=0}^\infty b_n q^n \,.
$$
If we assume that $f(0,e) \in \BZ[[q]]$, this imposes a system of linear 
equations on the coefficients $a_n$ and $b_n$ of $A$ and $B$. In fact,
for fixed $e <0$, the system of equations $\coeff(f(0,e),q^j)=0$
for $j=-e^2/2-e/2, \dots, -2,-1$ is a triangular system
of linear equations with unknowns $b_j$ for $j=0,1,\dots,e^2/2-e/2-1$
where all diagonal entries of the coefficient matrix are $1$. For example,
we have:
$$
\left(
\begin{array}{cccccc}
 1 & 0 & 0 & 0 & 0 & 0 \\
 -2 & 1 & 0 & 0 & 0 & 0 \\
 -2 & -2 & 1 & 0 & 0 & 0 \\
 1 & -2 & -2 & 1 & 0 & 0 \\
 4 & 1 & -2 & -2 & 1 & 0 \\
 3 & 4 & 1 & -2 & -2 & 1 \\
\end{array}
\right) 
\left(
\begin{array}{c}
b_0 \\
b_1 \\
b_2 \\
b_3 \\
b_4 \\
b_5
\end{array}
\right)
=
\left(
\begin{array}{c}
 -a_{0} \\
 a_{0}-a_{1} \\
 2 a_{0}+a_{1}-a_{2} \\
 a_{0}+2 a_{1}+a_{2}-a_{3} \\
 -2 a_{0}+a_{1}+2 a_{2}+a_{3}-a_{4} \\
 -3 a_{0}-2 a_{1}+a_{2}+2 a_{3}+a_{4}-a_{5} \\
\end{array}
\right)
$$
It follows that $b_n$ is a $\BZ$-linear combination of $a_k$ for $k \leq n$.
This proves that the rank of the $\BZ[[q]]$-module $V_+$ is at most $1$.
Since $\ID \in V_+$ (as follows from the proof of Theorem \ref{thm.2}), it
follows that the rank of the $\BZ[[q]]$-module $V_+$ is exactly $1$.
This proves part (b) of Theorem \ref{thm.1}. 
\qed

\begin{corollary}
\label{cor.V+}
The above proof implies that $f \in V_+$ is uniquely determined by
its initial condition $f(0,0) \in \BZ[[q]]$. It follows that if 
$f,g \in V_+$, then
\be
\label{eq.fg00}
g(0,0)f(m,e)=f(0,0)g(m,e)
\ee
for all integers $m$ and $e$.
\end{corollary}


\subsection{Proof of triality: part (c) of Theorem \ref{thm.1}}
\label{sub.prop.S3}

In this section we prove part (c) of Theorem \ref{thm.1}. Equation 
\eqref{eq.Z3} concerns the following $\BZ/3$-action on $V$.


\begin{definition}
\label{def.Sf}
Consider the action $f \mapsto Sf$ on a function 
$f: \BZ^2 \longto \BZ((q^{1/2}))$ given by:
\be
\label{eq.Sf}
Sf(m,e)=(-q^{\frac{1}{2}})^{-e} f(e,-e-m) \,.
\ee
\end{definition}

\begin{proposition}
\label{prop.S3}
\rm{(a)} We have: $S^3=Id$.
\newline
\rm{(b)} If $f \in V$, then $Sf=f$, and of course, also $S^2f=f$.
\end{proposition}
Part (c) of Theorem \ref{thm.1} follows from part (b) of the above
proposition.

\begin{proof}(of Proposition \ref{prop.S3})
Part (a) is elementary. For part (b), assume that $f$ satisfies Equation 
\eqref{eq.rec1} for all $(m,e)$. Replace $(m,e)$ by $(e,-1-e-m)$ in
\eqref{eq.rec1} and we obtain that
\be
\label{eq.em1}
-f(e, -1 - e - m) + q^{-\frac{e}{2}} f(e, -e - m) + 
 q^{-\frac{1}{2} - \frac{e}{2} - \frac{m}{2}} f(1 + e, -1 - e - m)=0 \,.
\ee
Now, replace $f$ by $Sf$ in the left hand side of Equation \eqref{eq.rec1},
and compute that the result is given by 
$$
(-1)^{e+1}\left( -f(e, -1 - e - m) + q^{-\frac{e}{2}} f(e, -e - m) + 
 q^{-\frac{1}{2} - \frac{e}{2} - \frac{m}{2}} f(1 + e, -1 - e - m) \right)
$$ 
The above vanishes from Equation \eqref{eq.em1}.

Likewise, assume that $f$ satisfies Equation 
\eqref{eq.rec2} for all $(m,e)$. Replace $(m,e)$ by $(e,1-e-m)$ in
\eqref{eq.rec2} and we obtain that
\be
\label{eq.em2}
q^{\frac{1}{2} - \frac{e}{2} - \frac{m}{2}} f(-1 + e, 1 - e - m) 
+ q^{-\frac{e}{2}} f(e, -e - m) -  f(e, 1 - e - m)=0
\ee
Now, replace $f$ by $Sf$ in the left hand side of Equation \eqref{eq.rec2},
and compute that the result is given by 
$$
(-1)^{e+1}\left( q^{\frac{1}{2} - \frac{e}{2} - \frac{m}{2}} f(-1 + e, 1 - e - m) + q^{-\frac{e}{2}} f(e, -e - m) - 
 f(e, 1 - e - m) \right)
$$
It follows that if $f \in V$, then the above vanishes from Equation 
\eqref{eq.em2}. In other words, if $f \in V$ then $Sf \in V$. 
To conclude that $f=Sf$, it suffices 
to show (by part (a) of Theorem \ref{thm.1}) that $f(0,0)=(Sf)(0,0)$.
If $f(0,0)=A$, $f(0,1)=B$, using Remark \ref{rem.thm1a} we have:
$$
(Sf)(0,0)=f(0,0)=A \qquad (Sf)(0,1)=f(0,1)+q^{-\frac{1}{2}}f(1,-1)=
B + q^{-\frac{1}{2}} (- B q^{\frac{1}{2}})=0 \,.
$$
This concludes the proof of Proposition \ref{prop.S3}.
\end{proof}

\subsection{$\ID$ is a tetrahedron index}
\label{sub.isindex}

Observe that by its definition, 
$$
\ID(m,e)=\sum_{e \in \BZ} S(m,e,n) 
$$ 
is given by a one-dimensional sum of a {\em proper $q$-hypergeometric term}
(\cite{WZ,PWZ}) 
$$
S(m,e,n)=(-1)^n \frac{q^{\frac{1}{2}n(n+1)
-\left(n+\frac{1}{2}e\right)m}}{(q)_n(q)_{n+e}} \,.
$$
It follows by \cite{WZ} that $\ID(m,e)$ is $q$-holonomic in both variables
$m$ and $e$. Moreover, recursion relations for $\ID(m,e)$ can be found
by the creative telescoping method of \cite{WZ}. For instance, $S$
satisfies the recursion
\be
\label{eq.Scert2}
q^{\frac{e}{2}} S(m - 1, e, n)+ q^{-\frac{m}{2}} S(m, e - 1, n) -S(m,e,n)=0
\ee
which implies that $\ID$ satisfies Equation \eqref{eq.rec2}. To prove
Equation \eqref{eq.Scert2}, divide by $S(m,e,n)$ and use the fact that
$$
q^{\frac{e}{2}} \frac{S(m - 1, e, n)}{S(m,e,n)}=q^{e + n}
\qquad
 q^{-\frac{m}{2}} \frac{S(m, e - 1, n)}{S(m,e,n)}=1-q^{e + n} \,.
$$
The proof of Equation \eqref{eq.rec1} is similar. For an alternative
proof, using the quantum dilogarithm, see Section \ref{sec.QDL}.

\subsection{The degree of $\ID$}
\label{sub.degDI}

\begin{proof}(of Lemma \ref{lem.degID})
Consider the fan $F$ of $\BR^2$ with rays $(1,0)$, $(0,1)$ and $(1,-1)$.
Observe that the linear transformation $(m,e)\mapsto(e,-e-m)$ (which
appears in Definition \ref{def.Sf}) rotates the three cones of the fan
$F$, and preserves the piece-wise quadratic polynomial that appears in
Lemma \ref{lem.degID}. Since $\ID \in V$ (by Section \ref{sub.isindex}) 
and $V$ is pointwise invariant under $S$ (by Proposition \ref{prop.S3}), 
it suffices to compute $\d(m,e)$ when $(m,e)$ lies in the cone 
$m \leq 0, e \geq 0$. In that case, Equation \eqref{eq.ID} gives
$$
\ID(m,e)=\sum_{n=0}^\infty (-1)^n \frac{q^{\frac{1}{2}n(n+1)
-\left(n+\frac{1}{2}e\right)m}}{(q)_n(q)_{n+e}}
$$
If $\d(m,e,n)$ denotes the degree of the summand, using $m \leq 0$, $n \geq 0$
we get
$$
\d(m,e,n)=\frac{1}{2}\left(n(n+1)\right) -\left(n+\frac{1}{2}e\right)m 
\geq -\frac{em}{2}\,,
$$
with equality achieved uniquely at $n=0$. It follows that the degree of
$I(m,e)$ in this cone is given by $-em/2$. 
\end{proof}

\subsection{Proof of Theorem \ref{thm.2}}
\label{sub.thm.2}

First we show that $\ID$ satisfies the required equations: 
\begin{itemize}
\item[(a)] $\ID \in V$ from Section \ref{sub.isindex}. Lemma \ref{lem.degID}
and Equation \eqref{eq.degID} manifestly imply that
$\d(\ID(m,e)) \geq 0$ for all integers $m$ and $e$. Thus, $\ID \in V_+$.
Moreover, $\ID(0,0)=1+O(q)$.
\item[(b)]
$\ID$ satisfies the pentagon identity from Section 
\ref{sub.pentagon}.
\end{itemize}

It remains to show the uniqueness part in Theorem \ref{thm.2}. Suppose
$f \in V_+$ satisfies the pentagon and $f(0,0)(0) \neq 0$. Corollary
\ref{cor.V+} implies that $f(m,e)(q)=C(q)\ID(e,m)(q)$ for some $C(q)
\in \BQ((q))$. 
Consider the special pentagon for $f$ and $\ID$:
$$
f(0,0)^2=\sum_{e \in \BZ} f(0,e)^3 q^e 
\qquad \ID(0,0)^2=\sum_{e \in \BZ} \ID(0,e)^3 q^e
\,.
$$
It follows that $C(q)^2=C(q)^3$ and since $C(q) \neq 0$, we get $C(q)=1$.
This concludes the uniqueness part of Theorem \ref{thm.2}.
\qed


\section{Convergence of the 3D index}
\label{sec.convergence}

\subsection{Proof of Theorem \ref{thm.IMconv}}
\label{sub.pf.thm.IMconv}

In this section we prove Theorem \ref{thm.IMconv}.
We begin by a well-known lemma due to Farkas \cite{Ziegler}.

\begin{lemma}
\label{lem.farkas}
Fix finite collections $\calA=\{a_1,\dots,a_r\}$ and
$\calB=\{b_1,\dots,b_s\}$ of vectors in $\BR^N$. 
The following are equivalent:
\begin{itemize}
\item[(a)] there does not exist $v \neq 0$ such that $a_i \cdot v \geq 0$ 
for $i=1,\dots,r$ and $b_j \cdot v =0$ for $j=1,\dots,s$.
\item[(b)] $\calA \cup \calB$ span $\BR^N$ and there exist $\a_i >0$
for $i=1,\dots,r$ and $\ga_j \in \BR$ for $j=1,\dots,s$ such that 
$0=\sum_i \a_i a_i + \sum_j \ga_j b_j$.
\end{itemize}
\end{lemma}

\begin{proof}
(a) is equivalent to 

\begin{itemize}
\item[(c)]  there does not exist $v \neq 0$ such that $a_i \cdot v \geq 0$ 
for $i=1,\dots,r$ and $b_j \cdot v \geq 0$ for $j=1,\dots,s$ and
$(-b_j) \cdot v \geq 0$ for $j=1,\dots,s$.
\end{itemize}
(c) implies (b).
Let $C$ denote the cone spanned by $\calA \cup \calB \cup -\calB$.
(c) states that $C$ not contained in any half-space through the
origin. By Farkas' lemma \cite{Ziegler}, it follows that $C=\BR^N$. Thus,
$\calA \cup \calB \cup -\calB$ spans $\BR^N$ and $-\sum_i a_i \in C$.
(b) follows.

(b) implies (c): consider $v$ such that $a_i \cdot v \geq 0$ and $b_j \cdot
v =0$ for all $i,j$. We know there exist $\a_i>0$ and $\ga_j$ real such that
 $0=\sum_i \a_i a_i + \sum_j \ga_j b_j$. Taking inner product with $v$,
it follows that $0=\sum_i \a_i a_i \cdot v$. Since $\a_i >0$ and $ a_i \cdot v
\geq 0$ for all $i$, it follows that $ a_i \cdot v =0$ for all $i$. Thus,
$v$ is perpendicular to $\calA \cup \calB$ which is assumed to span $\BR^N$.
Thus $v=0$ and (c) follows.
\end{proof}

The next lemma concerns super-linear polynomial functions on a cone.
\begin{lemma}
\label{lem.pieceway}
Suppose $C$ is a closed cone in $\BR^r$ and $p: C \longto \BR$ is a polynomial
that satisfies $p(n x) \geq c_x n$ for $n >0$, $x \in C\setminus\{0\}$ 
and $c_x >0$. Then, there exists $c>0$ and $c'>0$ such that $p(x) \geq c|x|$ 
for all $x \in C$ with $|x| \geq c'$.
\end{lemma}

\begin{proof}
Let $S=\{x \in \BR^r \,| \,\,|x|=1\}$ denote the unit sphere and let
$p=\sum_{k=0}^d p_k(x)$ denote the decomposition of $p$ into homogeneous
polynomials $p_k$ of degree $k$. 
Since $p(n x)=\sum_k n^k p_k(x)$, it follows that that for every 
$x \in S \cap C$ there exists $i$ such that $p_j(x)=0$ for $j>i$ and $
p_i(x) >0$. In particular, $p_d: S \cap C \longto [0,\infty)$.

{\bf Case 1:} $p_d( S \cap C) \subset (0,\infty)$. By compactness,
$p_d(x) \geq c_0>0$ for $x \in S \cap C$ and $|p_k(x)| \leq c_k$ for
$x \in S \cap C$ and $k=1,\dots,d-1$. Thus 
$p(x) \geq c_0 |x|^d - \sum_{k=0}^{d-1}|x|^k c_k \geq c |x|$
for some $c>0$.

{\bf Case 2:} There exists $x \in S \cap C$ such that $p_d(x)=0$
and $p_{d-1}(x) >0$. Argue as above using the complement of a neighborhood
of $x$ where $p_1$ is strictly positive, and conclude the proof by induction
on the depth of a point.
\end{proof}

Consider the restriction 
\be
\label{eq.Irho}
I^{\rho}_{\mb M}(m,e)(q)=\sum_{n \in \BN}
q^{\frac{n}{2} v \cdot k_0}
\prod_{i=1}^r \ID(m_i-n b_i \cdot k_0, e_i+n a_i \cdot k_0)
\ee
of the sum that defines $I_{\mb M}$
on a ray $\rho=\BN k_0$ for $k_0 \in \BZ^r$, $k_0 \neq 0$.
Consider the union $R$ of the 3 rays in $\BR^2$ 
shown in Figure \ref{fig.3rays}.

\begin{figure}[htpb]
\begin{center}
\includegraphics[height=0.12\textheight]{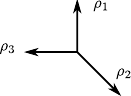}
\caption{The degree of the tetrahedron index.}
\label{fig.3rays}
\end{center}
\end{figure}

If $\rho=\BN k_0$ is a fixed ray, let $x=\mb B^T k_0=(x_1,\dots,x_s)$
and $y=\mb A^T k_0=(y_1,\dots,y_s)$.

\begin{lemma}
\label{lem.IMrho}
\rm{(a)} If $(-x_i,y_i) \not\in R$ for some $i=1,\dots,s$, then 
$I^{\rho}_{\mb M}(m,e)$ converges for all $m,e$.
\newline
\rm{(b)} If $(-x_i,y_i) \in R$ for all $i=1,\dots,s$. Then, there exist
$Q \in \{1,\dots,s\}\to\{1,2,3\}$ such that $(-x_i,y_i) \in \rho_{Q(i)}$
 for all $i=1,\dots,s$. Then $I^{\rho}_{\mb M}$ does not converge if and
only if 
all of the following inequalities hold:
\begin{subequations}
\label{eq.nok0}
\begin{align}
b_i \cdot k_0 & =0, &  a_i \cdot k_0 & \geq 0, & (-v) \cdot k_0 & \leq 0
& \text{if} \quad Q(i)=1 \\
(a_i-b_i) \cdot k_0 & =0, &  (-b_i) \cdot k_0 & \geq 0, & (-v+b_i) 
\cdot k_0 & \leq 0
& \text{if} \quad Q(i)=2 \\
(-a_i) \cdot k_0 & =0, &  (-a_i+b_i) \cdot k_0 & \geq 0, & (-v+a_i) 
\cdot k_0 & \leq 0 & \text{if} \quad Q(i)=3
\end{align}
\end{subequations}
\end{lemma}

\begin{proof}
(a) Without loss of generality, let us assume $m=e=0$. In that case,
the degree of the summand in Equation \eqref{eq.Irho} is given by
$$
n^2 \sum_{i=1}^s \d_2(-x_i,y_i) + n \sum_{i=1}^s \d_1(-x_i,y_i) +
\frac{n}{2} v \cdot k_0  
$$
where $\d_1$ and $\d_2$ are piece-wise quadratic and linear functions given
by Figure \ref{fig.delta12}.

\begin{figure}[htpb]
\begin{center}
\includegraphics[height=0.12\textheight]{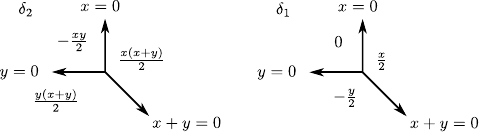}
\caption{Piece-wise quadratic and linear functions $\delta_2$ and $\delta_1$.}
\label{fig.delta12}
\end{center}
\end{figure}

If $(-x_i,y_i) \not\in R$ for some $i=1,\dots,s$, it follows that the degree
of the summand is a quadratic function of $n$ with nonvanishing leading
term, thus $I^{\rho}_{\mb M}$ converges.

(b) The above computation shows that  $I^{\rho}_{\mb M}(0,0)$ diverges
if and only if $\d_2(-x_i,y_i)=0$ for all $i=1,\dots,s$ and in addition
the coefficient of $n$ is less than or equal to zero. The first condition
is equivalent to $(-x_i,y_i) \in R$ for all $i$ and together with the second
one, they are equivalent to the inequalities \eqref{eq.nok0}.
\end{proof}

\begin{proof}(of Theorem \ref{thm.IMconv})
Lemma \ref{lem.pieceway} implies that $I_{\mb M}$ converges if and only 
$I^{\rho}_{\mb M}$ converges for all rays $\rho$. This is true since
the degre of the summand of $I_{\mb M}$ is a piece-wise quadratic 
polynomial. Lemma \ref{lem.IMrho} gives necessary and sufficient 
conditions for the convergence of $I^{\rho}_{\mb M}$. It remains to match
these conditions with the definition of an index structure on $\mb M$
using Lemma \ref{lem.farkas}.

The above discussion implies that $I_{\mb M}$ is convergent if and only if
for every $Q \colon \{1,\dots,s\}\to\{1,2,3\}$, 
there does not exist $k_0 \neq 0$ such that Equation 
\eqref{eq.nok0} holds. Assume for simplicity that $s=1$. 

{\bf Case 1:} If $Q(1)=1$ Inequality \eqref{eq.nok0}
and Lemma \ref{lem.farkas} implies that there exist $\a_1>0$ and $\ga_1$
real such that $v=\a_1 a_1 + \ga_1 b_1$. Define $\b_1=1-\a_1-\ga_1$.

{\bf Case 2:} If $Q(1)=2$ Inequality \eqref{eq.nok0}
and Lemma \ref{lem.farkas} implies that there exist $\a'_1>0$ and $\ga'_1$
real such that $v-b_1=\a'_1 (-b_1) + \ga'_1(a_1- b_1)$. Letting
$(\a_1,\b_1,\ga_1)=(\ga'_1,\a'_1,-\ga'_1-\a'_1+1) $ it follows that
$$
v=\a_1 a_1 + \ga_1 b_1, \qquad \b_1 > 0 \,.
$$

{\bf Case 3:} If $Q(1)=3$ Inequality \eqref{eq.nok0}
and Lemma \ref{lem.farkas} implies that there exist $\a'_1>0$ and $\ga'_1$
real such that $v-a_1=\a'_1 (-a_1 + b_1) + \ga'_1 (-a_1) $. Letting
$(\a_1,\b_1,\ga_1)=(1-\a'_1-\ga'_1,\ga'_1,\a'_1) $ it follows that
$$
v=\a_1 a_1 + \ga_1 b_1, \qquad \ga_1 > 0 \,.
$$
It follows that $\mb M$ admits an index structure.

The general case of $s$ follows as above. Indeed for each 
$Q: \{1,\dots,s\}\to \{1,2,3\}$, assume $(-x_i,y_i) \in \rho_{Q(i)}$
for $i=1,\dots,s$. Then $I^{\rho}_{\mb M}$ convergences if and only if
there exists $(\a,\b,\ga)$ that satisfies Equations 
\eqref{eq.index.structure} and inequalities \eqref{eq.index.ineq}.
This completes the convergence proof of Theorem \ref{thm.IMconv}.
$q$-holonomi\-city follows from the main theorem of Wilf-Zeilberger 
\cite{WZ}, using the fact that $I_{\mb M}(m,e)$ is a $2r$-dimensional 
sum of a proper $q$-hypergeometric summand.  
\end{proof}

\subsection{An independent proof of convergence for strict index 
structures}
\label{sub.independent.strict}

Theorem \ref{thm.IMconv} implies that $I_{\mb M}$ converges when 
$\mb M$ admits a strict index structure. In this section we give an 
independent proof of this fact without using the restriction of the summand
of the index to a ray.

\begin{proposition}
\label{prop.IMconv.strict}
If $\mb M$ supports a strict index structure, 
then $I_{\mb M}(m,e)(q) \in \BZ((q^{1/2}))$ 
is convergent for all $m,e \in \BZ^s$.
\end{proposition}

The proof of proposition \ref{prop.IMconv.strict} requires some lemmas.

\begin{lemma}
\label{lem.Iconv}
Fix positive real numbers $\a,\b >0$ with $\a+\b < 1/2$ and let 
$\ga = \min\{\a,\b,1/2-\a-\b\}$. Then for all for all integers
$m,e$ we have:
$$
\d(\ID(m,e) q^{-\b m + \a e}) \geq \ga \max\{|m|,|e|,|m+e|\} 
\,.
$$
\end{lemma}

\begin{proof}(of Lemma \ref{lem.Iconv})
Let $L_+(m,e)=\max\{|m|,|e|,|m+e|\}$. $L_+(m,e)$ is a piece-wise linear
function given by Figure \ref{fig.Lplus}.

\begin{figure}[htpb]
\begin{center}
\includegraphics[height=0.12\textheight]{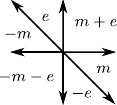}
\caption{A piece-wise linear function $L_+$.}
\label{fig.Lplus}
\end{center}
\end{figure}

With the notation of Lemma \ref{lem.degID}, we need to show that
\be
\label{eq.degDL}
\d(m,e)-\b m + \a e \geq \ga L_+(m,e) \,.
\ee
First, consider the three rays of $\d(m,e)$:

\begin{center}
\begin{tabular}{|c|c|c|} \hline
ray & left hand side of \eqref{eq.degDL} &
right hand side of \eqref{eq.degDL}  \\ \hline
$m=0, e \geq 0$ & $\a e$ & $\ga e$ \\ \hline
$e=0, m \leq 0$ & $ -\b m$ & $-\ga m$   \\ \hline
$m=-e \geq 0$ &  $m(1/2-\a-\b)$ & $\ga m$ \\ \hline
\end{tabular}
\end{center}

This proves inequality \eqref{eq.degDL} in the three rays and shows that
the choice of $\ga$ is optimal. Now, in the interior of each of the $6$ cones
of linearity of $L_+$, $\d(m,e) -\b m + \a e$ is given by a quadratic polynomial
of $m,e$. The degree $2$ (resp. $1$) part of this polynomial is always 
greater than or equal to $1/2 L_+(m,e)$ (resp. $(1/2-\ga) L_+(m,e)$)
by a case computation. For example, in the cone 
$m \geq 0, e \leq 0, e+m \geq 0$ with rays $\BR_+(1,0)$ and $\BR_+(1,-1)$
we have $\d(m,e)=m(m+e)/2+m/2$ and $L_+(m,e)=m$ and
\begin{eqnarray*}
\d(m,e)-\b m + \a e &=& \frac{m(e+m)}{2}+\frac{m}{2}-\b m+\a e
\\ & \geq & \frac{m}{2} + \frac{m}{2} -\b m+\a e
\\ &= & (1-\b-\a) m + \a(m+e) 
\\ & \geq & (1-\b-\a) m \geq (1-\b-\a) L_+(m,e) 
\end{eqnarray*}
The other cases are similar. 
\end{proof}

The next lemma is well-known \cite{Ziegler}.

\begin{lemma}
\label{lem.Pv}
Consider the convex polytope $P$ in $\BR^r$ defined by
$$
P=\{ x \in \BR^r \, | v_i \cdot x \leq c_i \quad i=1,\dots s \} \,
$$
where $v_i \in \BR^r$ and $c_i \in \BR$ for $i=1,\dots,s$. Then $P$
is compact if and only if the linear span of the set $\{v_i| i=1,\dots,s\}$
is $\BR^r$ and $0$ is a $\BR_{\geq 0}$-linear combination of elements of
$\{v_i| i=1,\dots,s\}$.
\end{lemma}

\begin{proof}(of Proposition \ref{prop.IMconv.strict})
Let $a_i$ and $b_i$ for $i=1,\dots,s$ denote the columns of $(\mb A|\mb B)$.
If $\mb M$ admits a strict index structure then there exist $\a_i, \ga_i >0$
that satisfy $a_i + \ga_i < 1$ for all $i$ such that
$$
\sum_{i=1}^s \a_i a_i + \ga_i \b_i = \nu \,.
$$
It follows that
$$
I_{\mb M}(m,e)(q)=\sum_{k \in \BZ^r}
\prod_{i=1}^s I(m_i-b_i \cdot k, e_i+a_i \cdot k)(q) 
q^{\frac{\b_i}{2} b_i \cdot k + \frac{\a_i}{2} a_i \cdot k} \,. 
$$
Applying Lemma \ref{lem.Iconv}, it follows that for every $k \in \BZ^d$,
the degree of the summand is bounded below by
$$
\sum_{i=1}^s \left( \b_i m_i - \a_i e_i \right)
+ \ga' \sum_{i=1}^s \left( |-m_i+b_i \cdot k| +|e_i+a_i \cdot k| \right)  
\,.
$$
Now, Lemma \ref{lem.Pv} and admissibility implies that for fixed $N_0$,
there are finitely many $k \in \BZ^d$ such that the above degree is
less than $N_0$. Proposition \ref{prop.IMconv.strict} follows. 
\end{proof}


\section{Invariance of the 3D index under $2\leftrightarrow
3$ moves and $2\leftrightarrow 0$ moves}
\label{sec.23}

\subsection{Invariance under the $3\to 2$ move}
\label{sub.3-2}

Consider two ideal triangulations $\calT$ and $\wt{\calT}$ with $N$ and $N+1$
tetrahedra, respectively, related by a $2-3$ move shown in Figure 
\ref{fig.23too}.

\begin{figure}[htpb]
\begin{center}
\includegraphics[height=0.20\textheight]{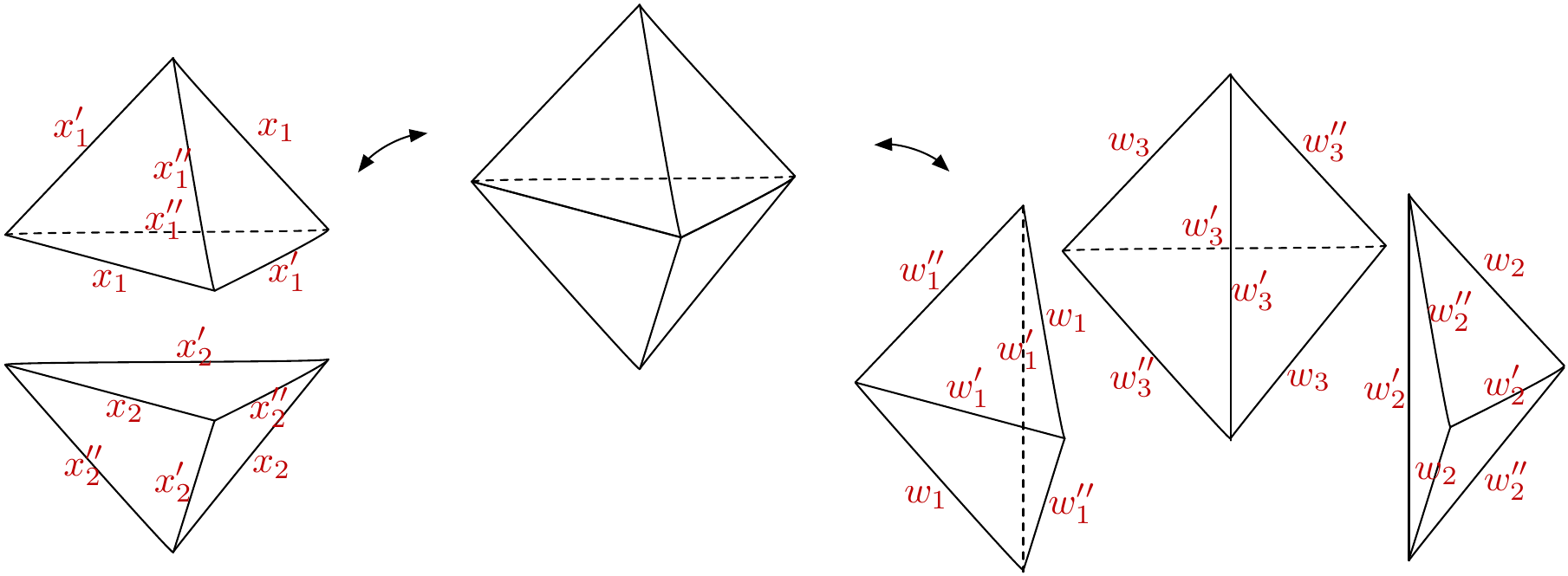}
\caption{A $2-3$ move.}
\label{fig.23too}
\end{center}
\end{figure}

The above figure matches the conventions of \cite[Sec.3.6]{DG}.
For a variable, matrix or vector $f$ associated to $\calT$, we will denote
by $\wt{f}$ the corresponding variable, matrix or vector associated to 
$\wt{\calT}$. Let us use variables $(Z,Z',Z'')$ 
and $(\wt Z,\wt Z',\wt Z'')$ to denote the angles of
$\calT$ and and $\wt{\calT}$ respectively, where
\be 
Z := (X_1,X_2,Z_3,...,Z_N)\,,\qquad 
\tilde Z := (W_1,W_2,W_3,Z_3,...,Z_N)\,.
\ee
We fix a quad type assigning these variables to $\calT$ and $\wt{\calT}$
as in Figure \ref{fig.23}. When calculating the Neumann-Zagier matrices,
we will assume that we keep the edge equation which comes from the internal
edge of the 2-3 bipyramid.

There are nine linear relations among the shapes of the tetrahedra involved 
in the move; three come from adding dihedral angles on the equatorial edges 
of the bipyramid:
\be 
\label{eq.s3}
W_1'=X_1+X_2\,,\qquad W_2'=X_1'+X_2''\,,\qquad 
W_3'= X_1''+X_2'\,, 
\ee
and six from the longitudinal edges:
\be 
\label{eq.s6}
\begin{array}{c@{\qquad}c@{\qquad}c}
X_1=W_2+W_3''\,,& X_1'=W_3+W_1''\,,&  X_1''=W_1+W_2''\,, \\[.2cm]
X_2 = W_2''+W_3\,,& X_2'=W_1''+W_2\,,& X_2''=W_3''+W_1\,.
\end{array}
\ee
Moreover, due to the central edge of the bipyramid, there is an extra 
gluing constraint in~$\wt \calT$:
\be
\label{eq.glue3} 
W_1'+W_2'+W_3' = 2\pi i\,. 
\ee
Let $\GA(\calT)$ and $\mathrm{A}(\calT)$ 
denote, respectively, the sets of generalized and strict angle structures 
of $\calT$.

\begin{lemma}
\label{lem.mu.2.3}
Consider the map:
\be
\label{eq.3-2.angles}
\mu_{3\to 2} \colon \GA(\wt{\calT}) \to \GA(\calT),
\qquad
\mu_{3\to 2} (\wt Z,\wt Z',\wt Z'') = (Z,Z',Z'') 
\ee
defined by Equations \eqref{eq.s6}. It induces a map
$$
\mu_{3\to 2} \colon \mathrm{A}(\wt{\calT}) \to \mathrm{A}(\calT)
$$
\end{lemma}

\begin{proof}
To check that $\mu_{3 \to 2}$ is well-defined
we need to show that $X_i+X_i'+X_i''=1$ is satisfied for $i=1,2$, assuming that 
Equation \eqref{eq.glue3} holds and 
$W_i+W_i'+W_i''=1$ for $i=1,2,3$. This is easy to check. If 
$(\wt Z,\wt Z',\wt Z'') \in \BR_+^{3(N+1)}$ (where $\BR_+$ is the set of 
positive real numbers), it is manifest from the definition that 
$(Z,Z',Z'') \in \BR_+^{3N}$. In other words, $\mu_{3\to 2}$ sends 
strict angle structures on $\wt{\calT}$ to those on $\calT$.
\end{proof}
This proves the first part of Proposition \ref{prop.3.2.index}.
To prove the remaining part, we study how are the gluing equation matrices
of $\calT$ and $\wt{\calT}$ related. Let $(\bar{\mb A}|\bar{\mb B}|\bar{\mb  C})$ 
denote the matrix of exponents of the gluing equations of $\calT$. We will 
use column notation and write
$$
\bar{\mb A} = (\bar a_1,\;\bar a_2,\;\bar a_i)\,,\quad 
\bar{\mb B} = (\bar b_1,\;\bar b_2,\;\bar b_i)\,,\quad
\bar{\mb C} = (\bar c_1,\;\bar c_2,\;\bar c_i)\,,
$$
where $\bar a_i$ meaning $(\bar a_3,\bar a_4,...,\bar a_N)$ and similarly for 
$\bar b_i$ and $\bar c_i$. Eliminating the $Z'$ variables, we obtain
$$
\mb A = \bar{\mb A} - \bar{\mb B}, \qquad
\mb B = \bar{\mb C} - \bar{\mb B} \,.
$$
In other words, 
\be
\label{eq.a1b1}
(a_1, a_2, a_i)=(\bar a_1 - \bar b_1, \bar a_2 - \bar b_2, \bar a_i - \bar b_i)
\quad
(b_1, b_2, b_i)=(\bar c_1 - \bar b_1, \bar c_2 - \bar b_2, \bar c_i - \bar b_i)
\,.
\ee
To compute the corresponding matrices of $\wt{\calT}$, use

\begin{align*}
2 &= \bar a_1 X_1 + \bar a_2 X_2 + \bar a_i Z_i 
+ \bar b_1 X'_1 + \bar b_2 X'_2 + \bar b_i Z'_i +
\bar c_1 X''_1 + \bar c_2 X''_2 + \bar c_i Z''_i \\
&= \bar a_1 (W_2+W_3'') + \bar a_2 ( W_2''+W_3) + \bar a_i Z_i \\
&\quad + \bar b_1 (W_3+W_1'') + \bar b_2 (W_1''+W_2) + \bar b_i Z'_i  \\
&\quad + \bar c_1(W_1+W_2'') + \bar c_2 (W_3''+W_1) + \bar c_i Z''_i \,.
\end{align*}
Collecting the coefficients of $\tilde Z, \tilde Z', \tilde Z''$ it
follows that the matrix of exponents of the gluing equations of $\wt{\calT}$
is given by
\begin{align*}
\wt {\bar{\mb A}} 
& = \begin{pmatrix} 0 & 0 & 0  & 0 \\ 
\bar c_1+\bar c_2 & \bar a_1 + \bar b_2 & \bar a_2 + \bar b_1 & \bar a_i 
\end{pmatrix}\,, \qquad
\wt {\bar{\mb B}}
=\begin{pmatrix} 1 & 1 & 1 & 0 \\ 
0 & 0 & 0 & \bar b_i 
\end{pmatrix}\,, 
\\
\wt {\bar{\mb C}}
&=\begin{pmatrix} 0 & 0 & 0 & 0 \\ 
\bar b_1+\bar b_2 & \bar a_2+\bar c_1 & \bar a_1+\bar c_2 & \bar c_i 
\end{pmatrix} \,.
\end{align*}
Use a row operation via
$P= \begin{pmatrix} 1 & 0 \\ \bar b_1 + \bar b_2 & \mb I \end{pmatrix}$
it follows that
\begin{align*}
P \wt {\bar{\mb A}} 
& = \begin{pmatrix} 0 & 0 & 0  & 0 \\ 
\bar c_1+\bar c_2 & \bar a_1 + \bar b_2 & \bar a_2 + \bar b_1 & \bar a_i 
\end{pmatrix}\,, \qquad
P \wt {\bar{\mb B}}
=\begin{pmatrix} 1 & 1 & 1 & 0 \\ 
\bar b_1+\bar b_2 & \bar b_1+\bar b_2 & \bar b_1+\bar b_2 & \bar b_i 
\end{pmatrix}\,, 
\\
P \wt {\bar{\mb C}}
&=\begin{pmatrix} 0 & 0 & 0 & 0 \\ 
\bar b_1+\bar b_2 & \bar a_2+\bar c_1 & \bar a_1+\bar c_2 & \bar c_i 
\end{pmatrix} \,.
\end{align*}
Since $\wt{\mb A} =\wt{ \bar{\mb A}} - \wt{\bar{\mb B}}$ and 
$\wt{\mb B} =\wt{ \bar{\mb C}} -\wt{ \bar{\mb B}}$, the above combined
with Equation \eqref{eq.a1b1} implies that

\be 
\label{eq.wtAB}
P \wt{\mb A} = \begin{pmatrix} -1 & -1 & -1  & 0 \\ 
b_1+b_2 & a_1 & a_2 & a_i \end{pmatrix}\,,\qquad P \wt{\mb B} 
=\begin{pmatrix} -1 & -1 & -1 & 0 \\ 0 & a_2+b_1 & a_1+b_2 & b_i 
\end{pmatrix}\,,
\ee
Since the 3D index is invariant under row operations (see Remark 
\ref{rem.GLr}), Equation \eqref{eq.wtAB} and the pentagon identity
\eqref{eq.pentagon} concludes that $I_{\wt{\calT}}=I_{\calT}$.
\qed

\subsection{Invariance under the $2\to 3$ move}
\label{sub.2-3}

In this section we will define what is a special angle structure on $\calT$,
and show the partial invariance of the 3D index under a $2\to 3$ move.
We will use the same notation as in Section \ref{sub.3-2}.
To define a map $(Z,Z',Z'') \mapsto (\wt Z,\wt Z',\wt Z'')$, we need to
solve for $W_i,W_i',W_i''$ for $i=1,2,3$ in terms of $X_i,X_i',X_i''$
for $i=1,2$ using Equations \eqref{eq.s3} and \eqref{eq.s6}. The answer
involves one free variable (say, $W_1$) and it is given by
\begin{subequations}
\label{eq.W1free}
\begin{align}
(W_1,W_2,W_3) &=
(W_1,W_1 + X_1 + X_2 + X_2''-1, W_1  + X_1 + X_2 + X_1' -1)
\\
(W_1',W_2',W_3')&=
(X_1 + X_2, X_1'-X_2-X_2''+1,-X_1-X_1'+X_2'+1)
\\
(W_1'',W_2'',W_3'')&= 
(-W_1 + 1 -X_1 -X_2, -W_1 -X_1 -X_1'+1, -W_1-X_2-X_2'+1) 
\end{align}
\end{subequations}
If $(Z,Z',Z'')$ is a strict angle structure on $\calT$, then 
$(\wt Z,\wt Z',\wt Z'')$ is a strict angle structure if and only if
Equations \eqref{eq.W1free} have a strictly positive solution. It is
easy to see that this is equivalent to the following condition
\be
\label{eq.special.angle}
X_1 + X_2 < 1, \qquad X_1''+X_2' < 1, \qquad X_1'+ X_2'' < 1 \,.
\ee
These conditions are precisely equivalent to the $W_1', W_2', W_3' <1$, 
as follows by Equation \eqref{eq.s3}. I.e., a special  strict angle structure
is an angle structure such that all angles of the bipyramid are less than
$1$.

\begin{definition}
\label{def.special.index}
We will say that $(Z,Z',Z'')$ is a {\em special strict angle structure}
on $\calT$ if the inequality \eqref{eq.special.angle} is satisfied.
\end{definition}
Let $\mathrm{A}^{\mathrm{sp}}(\calT)$ denote the set of special strict
angle structures on $\calT$. Then, we have a map (more precisely, a section of 
$\mu_{3 \to 2}$)
$$
\mu_{2\to 3} \colon \mathrm{A}^{\mathrm{sp}}(\calT) \to A(\wt{\calT}),
\qquad
\mu_{2\to 3}(Z,Z',Z'')=(\wt Z,\wt Z',\wt Z'') \,.
$$
The conclusion is that if $\calT$ admits a special strict angle structure,
then so does $\wt{\calT}$. In that case, $I_{\calT}$ and $I_{\wt{\calT}}$ both
exist. An application of the pentagon identity as in Section \ref{sub.3-2}
implies that $I_{\calT}=I_{\wt{\calT}}$.
\qed

\subsection{An ideal triangulation of {\tt m136}}
\label{sub.m136}

Let $\calT$ denote the ideal triangulation \cite[Ex.7.7]{HRS} of the
1-cusped census manifold {\tt m136} using 7 tetrahedra. Its gluing equation
matrices around the edges are given by:
{\tiny
\begin{align*}
\bar{\mb A} &=\left(
\begin{array}{ccccccc}
 1 & 0 & 0 & 0 & 1 & 1 & 1 \\
 0 & 0 & 0 & 1 & 0 & 1 & 0 \\
 0 & 0 & 1 & 0 & 0 & 0 & 0 \\
 0 & 1 & 0 & 0 & 0 & 0 & 0 \\
 1 & 0 & 0 & 0 & 0 & 0 & 0 \\
 0 & 1 & 1 & 1 & 0 & 0 & 1 \\
 0 & 0 & 0 & 0 & 1 & 0 & 0 \\
\end{array}
\right), &
\bar{\mb B} &=
\left(
\begin{array}{ccccccc}
 0 & 0 & 0 & 1 & 0 & 0 & 1 \\
 1 & 0 & 0 & 0 & 1 & 0 & 0 \\
 0 & 1 & 1 & 1 & 1 & 1 & 0 \\
 1 & 0 & 0 & 0 & 0 & 0 & 1 \\
 0 & 1 & 0 & 0 & 0 & 0 & 0 \\
 0 & 0 & 0 & 0 & 0 & 0 & 0 \\
 0 & 0 & 1 & 0 & 0 & 1 & 0 \\
\end{array}
\right), &
\bar{\mb C}&=
\left(
\begin{array}{ccccccc}
 1 & 0 & 0 & 1 & 0 & 1 & 1 \\
 0 & 0 & 0 & 0 & 1 & 1 & 0 \\
 1 & 2 & 0 & 1 & 0 & 0 & 0 \\
 0 & 0 & 1 & 0 & 0 & 0 & 0 \\
 0 & 0 & 1 & 0 & 1 & 0 & 0 \\
 0 & 0 & 0 & 0 & 0 & 0 & 0 \\
 0 & 0 & 0 & 0 & 0 & 0 & 1 \\
\end{array}
\right)
\end{align*}
}
A generalized angle structure is a solution to Equations \eqref{eq.angle}.
In our example, the set of generalized angle structures $\GA(\calT)$ 
is an affine $8$-dimensional subspace of $\BR^{21}$ and
the intersection $\SA(\calT)=\GA(\calT) \cap [0,\infty)^{21}$ is the polytope
of semi-angle structures. {\tt Regina} \cite{Bu}
gives that $\SA(\calT)$ is the convex hull of the following set of
$11$ points $(\a_1,\b_1,\ga_1,\dots,\a_7,\b_7,\ga_7)$ in $\BR^{21}$.

{\tiny
$$
\left(
\begin{array}{ccc|ccc|ccc|ccc|ccc|ccc|ccc}
0 & 0 & 1 & 1 & 0 & 0 & 0 & 0 & 1 & 1 & 0 & 0 & 0 & 0 & 1 & 0 & 1 & 0 & 0 &
0 & 1 \\
0 & 1 & 0 & 0 & 1 & 0 & 0 & 0 & 1 & 1 & 0 & 0 & 1 & 0 & 0 & 0 & 1 & 0 & 1 & 0 & 0 
\\ 
1 & 0 & 0 & 1 & 0 & 0 & 0 & 1 & 0 & 1 & 0 & 0 & 0 & 0 & 1 & 0 & 1 & 0 & 0 & 1 & 0 
\\ 
1 & 0 & 0 & 1 & 0 & 0 & 0 & 0 & 1 & 1 & 0 & 0 & 0 & 1 & 0 & 0 & 1 & 0 & 0 & 0 & 1 
\\ 
1/2 & 1/2 & 0 & 1 & 0 & 0 & 0 & 1/2 & 1/2 & 1/2 & 1/2 & 0 & 0 & 0 & 1 & 0 & 1 & 0 & 1/2 & 0 & 1/2
\\ 
1/2 & 1/2 & 0 & 1 & 0 & 0 & 0 & 1/2 & 1/2 & 1/2 & 0 & 1/2 & 0 & 0 & 1 & 0 & 1 & 0 & 1/2 & 0 & 1/2
\\ 
1/2 & 1/2 & 0 & 1 & 0 & 0 & 1/2 & 0 & 1/2 & 1/2 & 1/2 & 0 & 0 & 0 & 1 & 0 & 1 & 0 & 0 & 0 & 1 
\\ 
1/2 & 1/2 & 0 & 1 & 0 & 0 & 1/2 & 0 & 1/2 & 1/2 & 0 & 1/2 & 0 & 0 & 1 & 0 & 1 & 0 & 0 & 0 & 1 
\\ 
1/2 & 1/2 & 0 & 1/2 & 0 & 1/2 & 0 & 0 & 1 & 1 & 0 & 0 & 1/2 & 0 & 1/2 & 0 & 1 & 0 & 1/2 & 0 & 1/2
\\ 
1/2 & 1/2 & 0 & 1/2 & 1/2 & 0 & 1/2 & 0 & 1/2 & 1 & 0 & 0 & 1/2 & 0 & 1/2 & 0 & 1 & 0 & 0 & 1/2 & 1/2
\\ 2/3 & 1/3 & 0 & 2/3 & 0 & 1/3 & 1/3 & 0 & 2/3 & 1 & 0 & 0 & 1/3 & 0 & 2/3 & 0 & 1 & 0 & 0 & 1/3 & 2/3
\end{array}
\right)
$$
}
A computation shows that if $(\a,\b,\ga) \in \SA(\calT)$, then 
$(a_6,b_6,c_6)=(t,1,-t)$ for some $t \in \BR$ which explains why $\calT$ has 
no strict angle structure. 
On the other hand, \cite[Example 7.7]{HRS} mention $\calT$ has a solution
$$
(z_1,\dots,z_6)=\left(2 i,-1+2 i,\frac{3}{5}+\frac{1}{5}i,-1,
\frac{1}{5}+\frac{2}{5}i,2,\frac{1}{2}+\frac{1}{2}i \right)
$$
of the gluing equations which recovers the complete hyperbolic structure
on {\tt m136}.

%

\subsection{An ideal triangulation of {\tt m064}}
\label{sub.m064}

There is an explicit triangulation of {\tt m064} that uses 7 ideal
tetrahedra, communicated to us by Henry Segerman. Its gluing equation matrices
are given by

{\tiny
\begin{align*}
\bar{\mb A} &=\left(
\begin{array}{ccccccc}
 2 & 0 & 1 & 0 & 0 & 0 & 0 \\
 0 & 0 & 0 & 0 & 0 & 0 & 0 \\
 0 & 0 & 0 & 1 & 0 & 0 & 0 \\
 0 & 2 & 1 & 1 & 1 & 1 & 0 \\
 0 & 0 & 0 & 0 & 0 & 0 & 0 \\
 0 & 0 & 0 & 0 & 0 & 0 & 1 \\
 0 & 0 & 0 & 0 & 1 & 1 & 1 \\
 \end{array}
\right), &
\bar{\mb B} &=
\left(
\begin{array}{ccccccc}
 0 & 2 & 0 & 0 & 0 & 0 & 1 \\
 2 & 0 & 0 & 1 & 0 & 0 & 0 \\
 0 & 0 & 1 & 0 & 0 & 0 & 1 \\
 0 & 0 & 0 & 0 & 1 & 1 & 0 \\
 0 & 0 & 0 & 0 & 1 & 1 & 0 \\
 0 & 0 & 1 & 0 & 0 & 0 & 0 \\
 0 & 0 & 0 & 1 & 0 & 0 & 0 \\
\end{array}
\right), &
\bar{\mb C}&=
\left(
\begin{array}{ccccccc}
 1 & 0 & 0 & 1 & 1 & 1 & 0 \\
 0 & 0 & 1 & 0 & 0 & 0 & 0 \\
 1 & 0 & 0 & 0 & 0 & 0 & 0 \\
 0 & 1 & 0 & 0 & 0 & 0 & 1 \\
 0 & 1 & 0 & 0 & 1 & 1 & 0 \\
 0 & 0 & 1 & 1 & 0 & 0 & 1 \\
 0 & 0 & 0 & 0 & 0 & 0 & 0 \\
\end{array}
\right)
\end{align*}
}
This triangulation has no semi-angle structure, and its gluing equations
has the following numerical shape solution 
{\small
$$
(1.60 + 0.34 i, 
0.74 + 0.40 i, 
0.86 - 0.33 i, 
1.68 + 0.39 i, 
0.51 + 0.54 i, 
0.51 + 0.54 i, 
-0.61 + 1.25 i)
$$
}
that gives rise to the discrete faithful representation
of {\tt m064}. An explicit computation shows that this triangulation admits
an index structure.


\subsection{An ideal triangulation with no index structure}
\label{sub.noindex}

Consider an ideal triangulation $\calT$ which contains an edge $e$
and a tetrahedron $\Delta_1$ such that goes around $e$ five times with shapes
$Z$, $Z'$, $Z'$, $Z''$ and $Z''$. Suppose that no other tetrahedron 
touches $e$. Then the equation for a generalized angle structure around $e$ 
reads
$$
\a + 2 \b + 2 \ga =2, \qquad \a + \b + \ga =1
$$
This forces $\a=0$ so no generalized angle structure has $\a >0$. 
Note that the corresponding gluing equations around the edge $e$ reads
$$
z (z')^2 (z'')^2 = 1, \qquad z z' z'' =-1, \qquad z'=(1-z)^{-1} \,
$$
which forces $z=1$. Thus the gluing equations have no non-degenerate
solution, i.e., no solution with shapes in $\BC\setminus\{0,1\}$.

More complicated examples can be arranged using special configurations of
two or more edges and tetrahedra. In all examples that we could generate
with no index structure, the triangulation is degenerate. 

Of course, the arguments of a shape solution to the gluing equations 
is a generalized angle structure. The latter, however, need not be an
index structure if some of the shapes are real, or have negative
imaginary part; see for instance the triangulation of {\tt m064} 
in Section \ref{sub.m064}.

\subsection{Invariance under the $2 \leftrightarrow 0$ move}
\label{sub.20move}

The next lemma implies the invariance of the index of an ideal triangulation
under a $2\leftrightarrow 0$ move. Such a move is also known as a pillowcase
move, described in detail in \cite[Sec.6]{GHRS}. 

\begin{lemma}
\label{lem.20move}
For integers $m,e,c$ we have
$$
\sum_e \ID(m,e)\ID(m,e+c) q^e=\delta_{c,0} \,.
$$
\end{lemma}

\begin{proof}
Equations \eqref{eq.QDL} and \eqref{IDQDL} imply that
$$
\sum_e \ID(m,e)x^e = 
\frac{(q^{-\frac{m}{2}+1}x^{-1})_\infty}{
(q^{-\frac{m}{2}}x)_\infty} \,.
$$
Since
$$
\frac{(q^{-\frac{m}{2}+1}x^{-1})_\infty}{
(q^{-\frac{m}{2}}x)_\infty} \cdot
\frac{(q^{-\frac{m}{2}+1}(q x^{-1})^{-1})_\infty}{
(q^{-\frac{m}{2}}(q x^{-1}))_\infty} =1 \,,
$$ 
it follows that 
$$
\sum_{e,e'} \ID(m,e) x^e \ID(m,e') q^{e'} x^{-e'}=1 \,.
$$
Therefore,
$$
\sum_{e,e': e-e'=c} \ID(m,e) \ID(m,e') q^{e'} =\delta_{c,0} \,. 
$$
This implies that
$$
\sum_{e'} \ID(m,e'+c) \ID(m,e') q^{e'}=\delta_{c,0} \,.
$$
The result follows.
\end{proof}

\subsection{Acknowledgment}
The author wishes to thank Greg Blekherman, Nathan Dunfield,
Christoph Koutschan, Henry Segerman and Josephine Yu for enlightening 
conversations. The author wishes to especially thank Tudor Dimofte for 
explaining the 3D index, and for his generous sharing of his ideas. 

The work was initiated during a Clay Conference in Oxford, UK. The author 
wishes to thank the Clay Institute and Oxford University for their hospitality.

\appendix 

\section{$\ID$ satisfies the pentagon identity}
\label{sub.pentagon}

\centerline{\it by Sander Zwegers} \bigskip

There are several proofs of the key pentagon identity of the 
tetrahedron index $\ID$. The proofs may use
an integral representation of the quantum dilogarithm, or  
$q$-holonomic recursion relations, or algebraic identities of
generating series of $q$-series of Nahm type \cite{GL2}.

\subsection{A generating series proof of the pentagon identity}
\label{sub.genseries}

In this section we will prove that $\ID$ satisfies the pentagon
identity using generating series. We will abbreviate the Pochhammer symbol
$$
(x;q)_\infty =\prod_{n=0}^\infty (1-xq^n)
$$
by $(x)_\infty=(x;q)_\infty$. The proof 
\begin{itemize}
\item
starts from an associativity identity
$$
\frac{(z_1z_2)_\infty}{(z_1)_\infty (z_2)_\infty} \cdot
\frac{(x_1z_1^{-1}q)_\infty (x_2z_2^{-1} q)_\infty}{
(x_1 x_2 z_1^{-1}z_2^{-1}q)_\infty}=
\frac{(x_1z_1^{-1}q)_\infty}{(z_1)_\infty} \cdot
\frac{(x_2z_2^{-1}q)_\infty}{(z_2)_\infty} \cdot
\frac{(z_1z_2)_\infty}{(x_1 x_2 z_1^{-1}z_2^{-1}q)_\infty}
$$
that uses $4$ additional variables $\{x_1,x_2,z_1,z_2\}$ in addition to
the $4$ variables\newline $\{m_1,m_2,e_1,e_2\}$,
\item
extracts coefficients with respect to $(z_1,z_2)$ and 
\item
specializes $(x_1,x_2)=(q^{-m_1},q^{-m_2})$. This last part is
not algebraic and requires to show convergence. The latter follows from
Corollary \ref{cor.convP}.
\end{itemize}
Let us now give the details. Consider 
\be
\label{eq.Fex}
F_e(x)=\sum_n (-1)^n \frac{q^{\frac{1}{2}n(n+1)} x^n}{(q)_n(q)_{n+e}}
\in \BZ[[x,q]]\, .
\ee
Here and below, summation is over the set of integers, with the 
understanding that $1/(q)_n=0$ for $n < 0$. 

We will show that 
\be
\label{eq.Fgen}
q^{e_1 e_2} F_{e_1}(q^{e_2} x_1)F_{e_2}(q^{e_1} x_2)
=\sum_{e_3} (x_1 x_2 q)^{e_3} F_{e_1+e_3}(x_1) F_{e_2+e_3}(x_2) 
F_{e_3}(x_1 x_2)
\ee 
in the ring $\BZ((x_1,x_2,q))$. Since 
$$
F_e(q^{-m})=q^{\frac{e m}{2}}\ID(m,e) \,,
$$
the substitution $(x_1,x_2)=(q^{-m_1},q^{-m_2})$ (which converges by
Corollary \ref{cor.convP}) implies the pentagon identity of Equation
\ref{eq.pentagon}.

\begin{lemma}
\label{lem.5identities}
For $|q|<1$ we have:
\begin{eqnarray*}
\frac{1}{(x)_\infty} &=& \sum_n \frac{x^n}{(q)_n} \qquad\qquad |x|<1
\\
(xq)_\infty &=& \sum_n (-1)^n \frac{q^{\frac{1}{2}n(n+1)} x^n}{(q)_n}
\\
\frac{(xy)_\infty}{(x)_\infty} &=&  \sum_n \frac{(y)_n x^n}{(q)_n} 
\qquad\qquad |x|<1
\\
\frac{(xy)_\infty}{(x)_\infty (y)_\infty } &=&  \sum_{r,s} 
\frac{q^{r s} x^r y^s}{(q)_r(q)_s }\qquad\qquad |x|<1,\,|y|<1
\\
\frac{(xq)_\infty(yq)_\infty}{(xyq)_\infty } &=&  \sum_{r,s} 
(-1)^{r+s} \frac{q^{\frac{1}{2}(r-s)^2+\frac{1}{2}(r+s)} x^r y^s}{(q)_r(q)_s }
\qquad\qquad |xyq|<1
\end{eqnarray*}
\end{lemma}

\begin{proof}
The first three identities are well-known and appear in \cite[Prop.\ 2]{Zagier-dilog}. 
The last two follow from the first three:
\begin{eqnarray*}
\sum_{r,s} 
\frac{q^{r s} x^r y^s}{(q)_r(q)_s } &=& 
\sum_r \frac{x^{r}}{(q)_r} \sum_s 
\frac{(q^r y)^s}{(q)_s }= \sum_r \frac{x^{r}}{(q)_r} \frac{1}{(q^ry)_\infty}
\\
&=&
\frac{1}{(y)_\infty} \sum_r \frac{(y)_r x^r}{(q)_r}=
\frac{(xy)_\infty}{(x)_\infty (y)_\infty }\, ,
\end{eqnarray*}
\begin{eqnarray*}
\sum_{r,s} 
(-1)^{r+s} \frac{q^{\frac{1}{2}(r-s)^2+\frac{1}{2}(r+s)} x^r y^s}{(q)_r(q)_s }
&=&
\sum_r \frac{(-1)^r q^{\frac{1}{2}r^2+\frac{1}{2}r} x^r}{(q)_r}
\sum_s \frac{(-1)^r q^{\frac{1}{2}s^2+\frac{1}{2}s} (q^{-r}y)^s}{(q)_s} \\
&=&
\sum_r \frac{(-1)^r q^{\frac{1}{2}r^2+\frac{1}{2}r} x^r}{(q)_r} (q^{1-r}y)_\infty
\\
&=& (yq)_\infty \sum_r \frac{(y^{-1})_r (xyq)^r}{(q)_r} =
\frac{(xq)_\infty(yq)_\infty}{(xyq)_\infty } \, .
\end{eqnarray*}
\end{proof}

\begin{remark}
\label{rem.5identities}
The identities of Lemma \ref{lem.5identities} also hold in the ring
$\BZ((x,y,q))$.
\end{remark}
Observe that $F_e(x)$ is an analytic function of $(x,q)$ when
$|q|<1$ and $x \in \BC$. With $|q|<1$ and $|y|<1$, Lemma 
\ref{lem.5identities} gives

\begin{eqnarray*}
\sum_e F_e(x)y^e &=&
\sum_n \frac{(-1)^n q^{\frac{1}{2}n^2+\frac{1}{2}n} x^n}{(q)_n}
\sum_e \frac{y^e}{(q)_{n+e}} \\
&=& \frac{1}{(y)_\infty} 
\sum_n \frac{(-1)^n q^{\frac{1}{2}n^2+\frac{1}{2}n} (xy^{-1})^n}{(q)_n}
= \frac{(x y^{-1}q)_\infty}{(y)_\infty} \,.
\end{eqnarray*}
Thus, the generating function of the left hand side of Equation 
\eqref{eq.Fgen} is

\begin{multline*}
\sum_{e_1,e_2} q^{e_1 e_2} F_{e_1}(q^{e_2} x_1)F_{e_2}(q^{e_1} x_2) z_1^{e_1}
z_2^{e_2}  \\
=\sum_{n_1,n_2} \frac{(-1)^{n_1+n_2}q^{\frac{1}{2}n_1^2+\frac{1}{2}n_2^2 
+ \frac{1}{2}n_1+ \frac{1}{2}n_2}x_1^{n_1} x_2^{n_2}}{(q)_{n_1}(q)_{n_2}}
\sum_{e_1,e_2} \frac{q^{e_1 e_2 + n_2 e_1 + n_1 e_2}z_1^{e_1}z_2^{e_2} 
 }{(q)_{n_1+e_1}(q)_{n_2+e_2}}
\\
=\frac{(z_1 z_2)_\infty}{(z_1)_\infty (z_2)_\infty}
\sum_{n_1,n_2} \frac{(-1)^{n_1+n_2}q^{\frac{1}{2}(n_1-n_2)^2
+ \frac{1}{2}n_1+ \frac{1}{2}n_2}(x_1 z_1^{-1})^{n_1} 
(x_2 z_2^{-1})^{n_2}}{(q)_{n_1}(q)_{n_2}}
\\
=\frac{(z_1z_2)_\infty}{(z_1)_\infty (z_2)_\infty} \cdot
\frac{(x_1z_1^{-1}q)_\infty (x_2z_2^{-1} q)_\infty}{
(x_1 x_2 z_1^{-1}z_2^{-1}q)_\infty} \,.
\end{multline*}

Likewise, 
the generating function of the right hand side of Equation 
\eqref{eq.Fgen} is the same

\begin{multline*}
\sum_{e_1,e_2} \left( \sum_{e_3} (x_1 x_2 q)^{e_3} F_{e_1+e_3}(x_1) F_{e_2+e_3}(x_2) 
F_{e_3}(x_1 x_2)\right) z_1^{e_1} z_2^{e_2} \\ 
=
\left(\sum_{e_1} F_{e_1}(x_1) z_1^{e_1}\right)
\left(\sum_{e_2} F_{e_2}(x_2) z_2^{e_2}\right)
\left(\sum_{e_3} F_{e_3}(x_1x_2) (x_1 x_2 z_1^{-1} z_2^{-1} q)^{e_3}\right)
\\
=\frac{(x_1z_1^{-1}q)_\infty}{(z_1)_\infty} \cdot
\frac{(x_2z_2^{-1}q)_\infty}{(z_2)_\infty} \cdot
\frac{(z_1z_2)_\infty}{(x_1 x_2 z_1^{-1}z_2^{-1}q)_\infty}\, .
\end{multline*}

The above identities for each side of Equation \eqref{eq.Fgen}
hold when $|q|<1$, $|z_1|<1$, $|z_2|<1$ and
$|x_1x_2z_1^{-1}z_2^{-1}q|<1$. Remark \ref{rem.5identities} implies that they also hold in the ring $\BZ((x_1,x_2,z_1,z_2,q))$. 
Extracting the coefficient of $z_1^{e_1} z_2^{e_2}$ from the above
concludes the proof of Equation \eqref{eq.Fgen}.
\qed

\subsection{A second proof of the pentagon identity}
\label{sub.secondpentagon}

In this section we give a second proof of the pentagon identity
using
\begin{equation} 
\label{eq.5qa}
\begin{split}
\frac1{(q)_m (q)_n} &=\underset{s+t=n}{\underset{r+s=m}{\sum_{r,s,t}}} 
\frac{q^{rt}}{(q)_r(q)_s(q)_t}\, ,\\
\frac{q^{mn}}{(q)_m (q)_n} &=\underset{s+t=n}{\underset{r+s=m}{\sum_{r,s,t}}} 
\frac{(-1)^s q^{\ha s^2-\ha s}}{(q)_r(q)_s(q)_t}= 
\sum_s \frac{(-1)^s q^{\ha s^2-\ha s}}{(q)_{m-s}(q)_{n-s}(q)_s}\, .
\end{split}
\end{equation}
The first identity is well-known \cite[Eqn.(13)]{Zagier-dilog}, 
and the second follows from 
the first by replacing $q$ by $q^{-1}$ and multiplying both sides by 
$(-1)^{m+n} q^{-\ha (m-n)^2 -\ha(m+n)}$. 

Using these equations, we will show here that
\begin{equation}\label{one}
\begin{split}
q^{e_1e_2} &\frac{(-1)^{n_1+n_2} q^{\ha n_1^2+\ha n_2^2+\ha n_1+\ha n_2+e_2n_1+e_1n_2}}{
(q)_{n_1} (q)_{n_2} (q)_{n_1+e_1} (q)_{n_2+e_2}}\\
&= \underset{r_2+r_3+e_3=n_2}{\underset{r_1+r_3+e_3=n_1}{
\sum_{r_1,r_2,r_3,e_3}}} \frac{ (-1)^{r_1+r_2+r_3} 
q^{\ha r_1^2+\ha r_2^2 +\ha r_3^2 +\ha r_1+\ha r_2 +\ha r_3 +e_3}}{
(q)_{r_1} (q)_{r_1+e_1+e_3}(q)_{r_2}(q)_{r_2+e_2+e_3}(q)_{r_3}(q)_{r_3+e_3}}.
\end{split}
\end{equation}
The sum on the right actually only has a finite number of non-zero terms, so there is no issue with 
convergence. If we multiply both sides with $x_1^{n_1} x_2^{n_2}$ and sum 
over all $n_1$ and $n_2$, then we again find
\begin{equation*}
q^{e_1e_2} F_{e_1} (q^{e_2} x_1) F_{e_2} (q^{e_1}x_2) 
= \sum_{e_3} (x_1x_2q)^{e_3} F_{e_1+e_3}(x_1) F_{e_2+e_3}(x_2) F_{e_3}(x_1x_2)\, .
\end{equation*}
To prove \eqref{one} we use Equations \eqref{eq.5qa} which give
\begin{equation*}
\begin{split}
&\frac{q^{(n_1+e_1)(n_2+e_2)}}{(q)_{n_1} (q)_{n_2} (q)_{n_1+e_1} (q)_{n_2+e_2}}\\
& \qquad = \underset{r_2+e_3=n_2}{\underset{r_1+e_3=n_1}{\sum_{r_1,r_2,r_3,e_3}}} 
\frac{(-1)^{r_3} q^{\ha r_3^2-\ha r_3 +r_1r_2}}{(q)_{r_1} (q)_{r_2} (q)_{n_1+e_1-r_3} 
(q)_{n_2+e_2-r_3} (q)_{r_3} (q)_{e_3}}\, .
\end{split}
\end{equation*}
Replacing $e_3$ by $e_3+r_3$ in this sum we get
\begin{equation*}
\begin{split}
&\frac{q^{(n_1+e_1)(n_2+e_2)}}{(q)_{n_1} (q)_{n_2} (q)_{n_1+e_1} (q)_{n_2+e_2}}\\
&\qquad = \underset{r_2+r_3+e_3=n_2}{\underset{r_1+r_3+e_3=n_1}{\sum_{r_1,r_2,r_3,e_3}}} 
\frac{(-1)^{r_3} q^{\ha r_3^2-\ha r_3 +r_1r_2}}{(q)_{r_1} (q)_{r_2} (q)_{n_1+e_1-r_3} 
(q)_{n_2+e_2-r_3} (q)_{r_3} (q)_{r_3+e_3}}\\
&\qquad = \underset{r_2+r_3+e_3=n_2}{\underset{r_1+r_3+e_3=n_1}{\sum_{r_1,r_2,r_3,e_3}}} 
\frac{(-1)^{r_3} q^{\ha r_3^2-\ha r_3 +r_1r_2}}{(q)_{r_1} (q)_{r_2} (q)_{r_1+e_1+e_3} 
(q)_{r_2+e_2+e_3} (q)_{r_3} (q)_{r_3+e_3}}\, .
\end{split}
\end{equation*}
Now multiplying both sides by $(-1)^{n_1+n_2} q^{\ha (n_1-n_2)^2 +\ha n_1 +\ha n_2}$ 
gives Equation \eqref{one}.


\section{The tetrahedron index and the quantum dilogarithm}
\label{sec.QDL}

Gukov-Gaiotto-Dimofte came up with the beautiful formula \eqref{eq.ID}for the
tetrahedron index from a Fourier transform of the quantum dilogarithm. 
For completeness, we include this relation here, taken from \cite{DGG2}.
The quantum dilogarithm of Faddeev and Kashaev is a fundamental building
block of quantum topology \cite{FK-QDL,K,K94}. 
The $q$-series version of this analytic function is given by
\be
\label{eq.QDL}
L(m,x,q)=\frac{(q^{-\frac{m}{2}+1}x^{-1})_\infty}{
(q^{-\frac{m}{2}}x)_\infty} \in \BZ((x))[[q^{1/2}]]
\ee
We claim that
\be
\label{IDQDL}
\sum_e I(m,e)(q) x^e = L(m,x,q) \,.
\ee
To prove this, use the definition of $I(m,e)$, shift $e$ to $e-n$ and use
the first two identities of Lemma \ref{lem.5identities}. We get
\begin{align*}
\sum_e I(m,e)(q) x^e &=
\sum_{n,e} (-1)^n \frac{q^{\frac{1}{2}n(n+1)
-\left(n+\frac{1}{2}e\right)m}x^e}{(q)_n(q)_{n+e}} \\
&=
\sum_{n,e} (-1)^n \frac{q^{\frac{1}{2}n(n+1)} 
\left(q^{-\frac{m}{2}} x^{-1}\right)^n
\left(q^{-\frac{m}{2}}x\right)^e}{(q)_n(q)_{e}} \\
&= \frac{(q^{-\frac{m}{2}+1}x^{-1})_\infty}{
(q^{-\frac{m}{2}}x)_\infty} \,.
\end{align*}
Each of the recursion relations \eqref{eq.rec1}, \eqref{eq.rec2}, 
\eqref{eq.rec1a} and \eqref{eq.rec2a} is equivalent to the corresponding
relation \eqref{eq.QDLrec1}-\eqref{eq.QDLrec2a} 
for the generating series $L(m,x,q)$

\begin{subequations} 
\be
\label{eq.QDLrec1}
(-1+ q^{-\frac{m}{2}} x^{-1}) L(m,x,q) + L(m+1,q^{\frac{m}{2}} x,q) =0 
\ee
\be
\label{eq.QDLrec2} 
(1- q^{-\frac{m}{2}} x^{-1}) L(m,x,q) + L(m-1,q^{\frac{m}{2}} x,q) =0
\ee
\be
\label{eq.QDLrec1a}
(1+x^2 -(q^{\frac{m}{2}}+q^{-\frac{m}{2}})) L(m,x,q) 
+x q^{\frac{m}{2}} L(m,q x,q) =0 
\ee
\be
\label{eq.QDLrec2a} 
L(m-2,x,q) -( L(m-1,q^{\frac{1}{2}} x,q)+L(m-1,q^{-\frac{1}{2}} x,q)-q^{1-m}
L(m-1,q^{-\frac{1}{2}} x,q) ) + L(m,x,q) =0
\ee
\end{subequations}
Equations  \eqref{eq.QDLrec1}-\eqref{eq.QDLrec2a} are easy to verify
using the fact that $L(m,q,x)$ is a proper hypergeometric function of
$(m,q)$. This gives an alternative proof of part (a) of Theorem \ref{thm.2}.

Observe finally that the recursions \eqref{eq.rec1} and \eqref{eq.rec2} have
a solution space of rank $2$. On the other hand, the 
recursions \eqref{eq.QDLrec1} and \eqref{eq.QDLrec2} have
a solution space of rank $1$.

\bibliographystyle{plain}
\bibliography{biblio}
\end{document}